\documentclass[11pt,reqno]{amsproc}
\usepackage[margin=1in]{geometry}
\usepackage{amsmath, amsthm, amssymb}
\usepackage[colorlinks=true, pdfstartview=FitV, linkcolor=blue,citecolor=blue, urlcolor=blue]{hyperref}
\usepackage[abbrev,lite,nobysame]{amsrefs}
\usepackage{times}
\usepackage[usenames,dvipsnames]{color}
\usepackage{graphicx}
\usepackage{subcaption}
\usepackage{mathtools}

\mathtoolsset{showonlyrefs=true}
\usepackage{dsfont}

\newcommand{\E}{{\mathbb E}}
\newcommand{\LL}{{\mathcal L}}

\newcommand{\cS}{{\mathcal S}}

\newcommand{\T}{{\mathbb T}}

\newcommand{\cA}{\mathcal A}
\newcommand{\cL}{\mathcal L}

\newcommand\eps{\varepsilon}
\newcommand\e{{\rm e}}

\newcommand\dd{{\rm d}}
\newcommand\uu {\boldsymbol{u}}

\newcommand\ddt{{\frac{\dd}{\dd t}}}
\newcommand\R {\mathbb{R}}
\newcommand\C {\mathbb{C}}
\newcommand\N {\mathbb{N}}

\newcommand\ZZ {{\mathbb Z}}
\renewcommand\l {\langle}
\renewcommand\r {\rangle}
\newcommand\de{{\partial}}
\newcommand\XX {\mathbf{X}}

\newcommand\xx {\boldsymbol{x}}
\newcommand\WW {\mathbf{W}}
\newcommand\bv {\boldsymbol{b}}
\newcommand\bx {\boldsymbol{x}}

\def\aa{{\alpha}}
\def\bb{{\beta}}
\def\cc{{\gamma}}

\newtheorem{proposition}{Proposition}[section]
\newtheorem{theorem}[proposition]{Theorem}
\newtheorem{corollary}[proposition]{Corollary}
\newtheorem{lemma}[proposition]{Lemma}
\theoremstyle{definition}

\newtheorem{remark}[proposition]{Remark}

\numberwithin{equation}{section}

\title[Homogenization and hypocoercivity for Fokker-Planck equations]{Homogenization and hypocoercivity for Fokker-Planck equations \\
driven by weakly compressible shear flows}

\author[M. Coti Zelati, G. A. Pavliotis]{Michele Coti Zelati and Grigorios A. Pavliotis}

\address{Department of Mathematics, Imperial College London, London, SW7 2AZ, UK}
\email{m.coti-zelati@imperial.ac.uk}
\email{g.pavliotis@imperial.ac.uk}

\subjclass[2000]{35B27, 35K15, 60J60, 76F25}

\keywords{Fokker-Planck equation, shear flows, homogenization, enhanced diffusion, hypocoercivity}

\begin{document}

\begin{abstract}
We study the long-time dynamics of two-dimensional linear Fokker-Planck equations driven by a drift that can be decomposed in the sum of a large
shear component and the gradient of a regular potential depending on one spatial variable. The problem can be interpreted as that of a passive
scalar advected by a slightly compressible shear flow, and undergoing small diffusion. For the corresponding stochastic differential equation, we 
give explicit homogenization rates in terms of a family of time-scales depending on the parameter measuring the strength of the incompressible perturbation. This is achieved by exploiting an auxiliary Poisson problem, and by computing the related effective diffusion coefficients. Regarding
the long-time behaviour of the solution of the Fokker-Planck equation, we provide explicit decay rates to the unique invariant measure by employing
a quantitative version of the classical hypocoercivity scheme. From a fluid mechanics perspective, this turns out to be equivalent to quantifying the phenomenon of enhanced diffusion for slightly compressible shear flows.
\end{abstract}


\maketitle

\section{Introduction}

Scalar transport is an important problem with many applications to, e.g. atmosphere/ocean science and engineering~\cites{CKRZ08, Golden_al_2020, kramer}. The evolution of the density of a passive tracer is governed by the advection-diffusion (Fokker-Planck) equation
\begin{equation}\label{e:adv-diff}
\de_t \rho= - \nabla \cdot (\bv \rho) + \kappa \Delta \rho,
\end{equation}
with an initial density $\rho(0,\xx)= \rho_0(\xx)$, where $\bv=\bv(\xx)$ denotes the (fluid) velocity field and $\kappa>0$ stands for the molecular diffusion coefficient. The stochastic differential equation corresponding to the advection-diffusion equation is
\begin{equation}\label{e:SDE}
\dd \XX(t) = \bv(\XX(t)) \dd t + \sqrt{2\kappa}\,\dd \WW(t).
\end{equation}
One is usually interested in the long-time, large-scale behaviour of the dynamics of~\eqref{e:adv-diff} or, equivalently, \eqref{e:SDE}. It is well-known~\cites{lions, PavlSt08} that, for periodic or random velocity fields, the dynamics of the passive scalar at large scales becomes diffusive and can be quantified by the effective diffusion tensor $D$~\cites{Pap95,PavlSt08,kramer}. The dependence of the diffusion tensor on the properties of the velocity field $\bv$, in particular in the asymptotic regime of small molecular diffusion, is a problem that has attracted a lot of attention in recent decades~\cites{Golden_al_2020, kramer}. More precisely, let $\bv$ be a smooth periodic vector field (which is the case that we will consider in this paper). If $e$ is an arbitrary unit vector in $\R^d$,  the rescaled process $X^e_{\eps}(t):=e \cdot \eps \XX(t/\eps^2)$ converges weakly in $C([0,T]; \R )$ to a Brownian motion with diffusion coefficient $D^e$, that is
\begin{equation}\label{e:limit}
X^e_{\eps}(t) \rightarrow \sqrt{2 D^e} W(t),\qquad \text{as } \eps\to0,
\end{equation}
where $D^e := e \cdot D e$, $D$ being the diffusion matrix. In the above we have assumed that the vector field $\bv$ is centered with respect to the invariant measure of the process $\XX(t)$ when restricted to the torus $\T^d=[0,2\pi)^d$, see Equation~\eqref{e:centering} below. This process is an ergodic Markov process with generator
\begin{equation}
\cL = \bv \cdot \nabla + \kappa \Delta.
\end{equation}
This is a partial differential operator on $\T^d$, equipped with periodic boundary conditions. The calculation of the diffusion coefficient along the $e$-direction  $D^e$ requires the solution of two PDEs on $\T^d$, the stationary Fokker-Planck equation and an appropriate Poisson equation, together with the calculation of an integral over the unit torus~\cite{PavlSt08}*{Ch. 13}. The stationary Fokker-Planck equation reads
\begin{equation}\label{e:stationary_FP}
\cL^* \rho_{\infty} =0,
\end{equation}
where $\cL^*$, the Fokker-Planck operator appearing in~\eqref{e:adv-diff}, is the $L^2(\T^d)$--adjoint of $\cL$; the Poisson equation is
\begin{equation}\label{e:poisson}
- \cL \phi^e = b^e :=\bv \cdot e. 
\end{equation}
The diffusion coefficient is given by the formula
\begin{equation}
D^e := \kappa \|e + \nabla \phi^e \|^2_{L^2_{\rho_{\infty}}} = \kappa \int_{\T^d} |e + \nabla \phi^e|^2 \rho_{\infty}(\xx) \, \dd\xx.
\end{equation}
The PDEs~\eqref{e:stationary_FP} and~\eqref{e:poisson} are equipped with periodic boundary conditions and we have assumed the centering condition
\begin{equation}\label{e:centering}
\int_{\T^d} \bv(\xx) \rho_{\infty}(\xx) \, \dd \xx = 0.
\end{equation}
The effect of a nonzero mean flow is studied in~\cites{thesis, MajMcL93} and it will not be considered in this paper. Convergence theorems of the form~\eqref{e:limit} can be proved using either PDE~\cite{lions} or probabilistic techniques~\cites{pardoux, bhatta}.

A question that has attracted a lot of interest, both from a mathematical and a computational perspective, is the calculation of the diffusion tensor for different types of vector fields $\bv$. Since it can be calculated analytically only in very few cases, e.g. for shear flows or for gradient flows in one dimension~\cite{PavlSt08}*{Sec. 13.6}, in most cases the best one can hope for is the derivation of estimates on the diffusion tensor and on its dependence on the parameters of the problem such as the molecular diffusivity $\kappa$. This problem has been studied in detail for two particular types of vector fields $\bv$ in~\eqref{e:SDE}, namely gradient flows $\bv = -\nabla V$ or divergence--free flows, for which $\nabla \cdot \bv = 0$. In the former case, where the SDE~\eqref{e:SDE} becomes
\begin{equation}
\dd \XX (t) = -\nabla V(\XX(t)) \, \dd t  +\sqrt{2 \kappa} \dd \WW(t),
\end{equation}
it is well known that the diffusion is always depleted~\cite{PavlSt08}*{Ch. 13}, namely
\begin{equation}
D^e \leq \kappa,
\end{equation}
for all directions $e$. In fact, when $\kappa \ll 1$, the diffusion coefficient becomes exponentially small in $\kappa$~\cite{CampPiatn2002}:
$$
D^e \sim C_1 \e^{- C_2/ \kappa}, \quad \kappa \ll 1.
$$
On the other hand, when the vector field is divergence--free, then diffusion is always enhanced~\cite{kramer}, \cite{PavlSt08}*{Ch. 13}. Furthermore, both lower and upper bounds for the diffusion coefficient are known:
\begin{equation}
\kappa \leq D^e \leq \kappa + \frac{1}{\kappa}, \quad \kappa \in (0, +\infty).
\end{equation}
The asymptotic behavior of the diffusion coefficient in the limit as $\kappa \rightarrow 0$ depends on the detailed properties of the vector field $\bv$ and can be quite different in different directions of $\R^d$~\cite{kramer}. This is also reflected on the scaling of the relevant time scales, e.g. the diffusive time scale; this scaling depends crucially on the divergence-free vector field, i.e on whether it has open or closed streamlines. A detailed study of this is presented in~\cite{Fannjiang02}.

Much less is known about the diffusion coefficient for flows that are neither gradient nor divergence--free. Homogenization problems for compressible flows have been studied in a few papers~\cites{McLaughlinForest99, vergassola}. However, the problem of the derivation of rigorous estimates on the diffusion coefficient for periodic vector fields that are neither gradient nor divergence-free has not been addressed yet. This is precisely the problem that we address in this paper; in particular, we obtain quantitative information on the effect of compressible perturbations of divergence-free flows on the effective diffusion coefficient, and we also study enhanced dissipation rates, for compressible perturbations of shear flows. Our analysis is based on recently developed techniques~\cites{BW13,BCZ15,CZ19}, including the theory of hypocoercivity~\cite{Villani09}. 

Before discussing in detail the framework that we will consider in this paper we present our main results, we mention a couple of related problems. 
%
%
\medskip 

\noindent $\diamond$ \emph{Homogenization for Inertial Particles.}
The problem of homogenization and enhanced dissipation for velocity fields that are not divergence-free arises naturally in the study of inertial particles~\cites{PavlStuZyg07, PavlStZyg09, PavlStBan06, PavSt05b}. The equations of motion for inertial particles, written in non-dimensional form, read~\cite{RV20}
\begin{equation}\label{e:inertia}
\mbox{St} \, \ddot{\XX} = \left( \bv(t,\XX) - \XX \right) +\mbox{St} \beta D_t   \bv(t,\XX) + \sqrt{2  Pe^{-1} } \dot{\WW},
\end{equation}
where $\mbox{St}$ denotes the Strouhal number, $\beta$ the fluid density and $D_t = \partial_t + \uu \cdot \nabla$ the material derivative. Rigorous homogenization results for dynamics of the form~\eqref{e:inertia} were obtained in~\cites{HP04, PavlStuZyg07, PavSt05b}. In recent work~\cite{RV20} it was shown that in the small inertia (small Strouhal number) limit, the dynamics~\eqref{e:inertia} reduces to a passive tracer equation of the form~\eqref{e:SDE} in a modified velocity field that is no longer incompressible. The velocity field that the inertial particles experiences is 
\begin{equation}
\uu_e = \uu - \mbox{St}(1 - \beta) D_t \uu.
\end{equation}
In particular, for time-independent flows we have that $\nabla \cdot \uu_e = - \mbox{St}(1 - \beta) \nabla \cdot \big((\uu \cdot \nabla) \uu \big)$. Even though this quantity vanishes for shear flows, the flows studied in this paper, we believe that the connection between the study of inertial particles and the problem of homogenization and enhanced dissipation for compressible flows~\cite{vergassola} is an interesting one and we plan to investigate this further in future work. Some preliminary numerical experiments are presented in Section~\ref{sec:numerics}.

%
%
\medskip 

\noindent $\diamond$ \emph{Nonreversible Langevin Samplers.} 
A fundamental problem in statistics and in computational statistical mechanics is that of sampling from a probability measure $\pi(\dd\xx) = \frac{1}{Z} \e^{-V(\xx)^2} \, \dd\xx$ that is known up to the normalization constant. A standard approach to sampling from $\pi(\dd\xx)$ is to consider dynamics that is ergodic with respect to this measure. The natural choice is that of the overdamped Langevin dynamics
\begin{equation}\label{e:langevin}
\dd \XX(t) = - \nabla V(\XX(t)) \, \dd t + \sqrt{2} \, \dd \WW(t).
\end{equation} 
The rate of convergence of~\eqref{e:langevin} to the target distribution $\pi(\dd\xx) = \frac{1}{Z} \e^{-V(\xx)} \, \dd\xx$ is given by the Poincar\'{e} (spectral gap) and logarithmic Sobolev inequalities~\cite{Bakry2014}, and it depends only on the properties of the potential function $V$. In order to speed up convergence to equilibrium and to reduce the asymptotic variance, a natural approach is to perturb the dynamics~\eqref{e:langevin} by adding a divergence-free perturbation $\frac{1}{\nu} \, \uu(\xx)$ with 
\begin{equation}\label{e:div_free}
\nabla \cdot(\uu e^{-V}) = 0
\end{equation}
that does not change the invariant measure:
\begin{equation}
\dd \XX^{\nu}(t) = \left(-\nabla V(\XX^{\nu}(t)) + \frac{1}{\nu} \uu(\XX^{\nu}(t)) \right) \, \dd t + \sqrt{2} \, \dd \WW(t).
\end{equation}
It is indeed possible to prove that the divergence-free perturbation accelerates convergence to the target distribution~\cites{Hwang_al1993,Hwang_al2005,LelievreNierPavliotis2013} and, in addition, that it reduces the asymptotic variance~\cite{DuncanLelievrePavliotis2016}; the asymptotic variance plays a role analogous to that of the effective diffusion coefficient, and given by the same formula, in terms of an appropriate Poisson equation, or equivalently, of the Green-Kubo formula~\cite{Pavliotis2010}. It is worth noting that the set-up considered in this paper is in essence the opposite to this, i.e. we are interested in analyzing the effect of reversible perturbations to incompressible (divergence-free) flows on the long time behaviour of advection-diffusion equations.  We also mention that, in order for the 
divergence-free condition \eqref{e:div_free} to be satisfied, it is sufficient for the velocity field $\uu$ to be divergence-free and orthogonal to $\nabla V$. In this paper we will consider the SDE~\eqref{e:SDE} for such velocity fields:
\begin{align}\label{eq:sde2}
\dd \XX(t) =\bv(\XX(t))\dd t + \sqrt{2} \, \dd \WW(t),
\end{align}
with suitable initial conditions, with
\begin{align}\label{eq:bv}
\bv(\xx)=\frac{1}{\nu}\uu(\xx)-\nabla V(\xx), \qquad \uu\cdot \nabla V=0,
\end{align} 
and where $\nu>0$ and $\uu:\T^d\to\R^d$ is divergence-free and satisfies the condition
\begin{align}\label{eq:grad2}
\uu\cdot \nabla V=0.
\end{align} 
Our goal is to investigate the long-time behavior of its solutions in the limit $\nu\to 0$.

\subsection{Enhanced diffusion}\label{sub:introenhanced}
The probability density  $\rho:[0,\infty)\times \T^d \to \R$ of the solution $\XX(t)$ of  \eqref{eq:sde2} satisfies the Fokker-Planck equation 
\begin{align}\label{FP:rho}
\de_t \rho+\frac{1}{\nu} \uu \cdot\nabla \rho= \nabla \cdot\left(\rho \nabla V +\nabla \rho\right), \qquad \int_{\T^d}\rho(t,\xx)\dd\xx=1.
\end{align}
As mentioned above, the unique invariant density $\rho_\infty$  of the dynamics~\eqref{FP:rho} is given by  
the Gibbs measure 
\begin{align}\label{eq:Gibbs}
\rho_\infty(\xx)=\frac{1}{Z} \e^{- V(\xx)}, \qquad Z=\int_{\T^d} \e^{-V(\xx)}\dd \xx.
\end{align}
In order to study the convergence rates to $\rho_\infty$ of solutions to the Fokker-Planck equation, it is convenient to normalize with respect to the invariant distribution 
$\rho_\infty$ and consider the unknown $h$ defined by
\begin{align}
h(t,\xx)=\frac{\rho(t,\xx)}{\rho_\infty(\xx)}-1, \qquad \xx\in \T^d.
\end{align}
Indeed, $h$ satisfies the backward Kolmogorov equation
\begin{align}\label{FP:h}
\de_t h+\frac{1}{\nu} \uu \cdot\nabla h=\Delta h-\nabla V\cdot\nabla h,\qquad h(0,\xx)=h^{in}(\xx),
\end{align}
where $h^{in}$ is defined in terms of the initial distribution function of the process $\XX(t)$ as 
\begin{align}
h^{in}(\xx)= \frac{\rho(0,\xx)}{\rho_\infty(\xx)}-1.
\end{align}
Note that since \eqref{FP:rho} conserves mass, the same is true for \eqref{FP:h} for the weighted mass
\begin{align}
\int_{\T^d}h(t,\xx)\rho_\infty(\xx)\dd \xx=0, \qquad \forall t\geq 0.
\end{align}
Define the $L^2$-weighted space 
\begin{align}\label{eq:L2space}
L^2_{\rho_\infty}=\left\{ f:\T^d\to \R, \quad \int_{\T^d} |f(\xx)|^2\rho_\infty(\xx)\dd \xx<\infty, \quad \int_{\T^d}f(\xx)\rho_\infty(\xx)\dd \xx=0 \right\},
\end{align}
endowed with the natural norm and scalar product
\begin{align}\label{eq:L2scalnorm}
\l f,g\r=\int_{\T^d} f(\xx) g(\xx)\rho_\infty(\xx)\dd \xx,\qquad \| f\|^2=\int_{\T^d} |f(\xx)|^2\rho_\infty(\xx)\dd \xx.
\end{align}
It is straightforward to check that the operator $\Delta -\nabla V\cdot\nabla$ is symmetric in $L^2_{\rho_\infty}$ and
\begin{align}
\l \Delta f -\nabla V\cdot\nabla f,g\r=-\l  \nabla f,\nabla g\r,
\end{align}
while, thanks to \eqref{eq:grad2},
we have
\begin{align}
\l \uu\cdot\nabla f,g\r=-\l  f,\uu\cdot\nabla g\r,
\end{align}
for sufficiently smooth functions $f,g$. 

In advection-diffusion equations, when we take  $V=0$ in \eqref{FP:h}, the enhancement of diffusive mixing
in passive tracers  by a fast incompressible flow was studied in great generality in \cites{CKRZ08, Zlatos2010} from a qualitative standpoint, and quantitatively in the more recent works \cites{BW13, BCZ15, CZDE18,CZD19,CZ19,CZDri19,IXZ19,FI19,WEI18}. 
In particular, a necessary and sufficient condition for diffusion enhancement is that
the operator $\uu\cdot\nabla$ has no eigenfunctions in the homogeneous Sobolev space $\dot{H}^1$.

The main result of \cite{CKRZ08} on qualitative enhanced diffusion still holds with the presence of the potential $V$, thanks to the Hilbert space setting illustrated above, providing a framework for slightly compressible perturbations of incompressible flows. In this paper, we address the issue of quantitative estimates, when the velocity field $\uu$ is a two-dimensional shear flow with simple critical points, and $V$ is a potential depending on one of the two variables only. To put it in terms of the advection-diffusion equation \eqref{e:adv-diff}, we study a slightly compressible velocity field $\bv$, in which the main incompressible part $\uu$ in \eqref{eq:bv} is a shear flow.
We devise enhanced diffusion estimates via hypocoercivity methods, and relate them to the diffusive time-scales of the process $\XX(t)$. The corresponding result for the so-called Kolmogorov flow was proven in \cite{BW13}, and later generalized in \cite{BCZ15} to all  incompressible shear flows with a finite number of critical points. 

%
%
\subsection{A general abstract framework}
The setting described in Section \ref{sub:introenhanced} can be put in a more general abstract fashion as follows, see~\cite{Villani09}*{Chapter 2}. Suppose that we are given a Gibbs measure~\eqref{eq:Gibbs}. Any smooth vector field on $\T^d$ admits the decomposition 
\begin{align}\label{e:helmholtz_rho}
\bv = \frac{1}{\nu} \uu+\nabla \ln \rho_\infty,
\end{align}
where $\nabla \cdot (\uu \rho_\infty ) = 0$ and $\rho_\infty$ is the solution of the stationary Fokker-Planck equation
\begin{align}~\label{e:fp_stationary}
- \nabla \cdot (\bv \rho_\infty) +  \Delta \rho_\infty =0.
\end{align}
In writing \eqref{e:helmholtz_rho} we have already normalized the various vector fields and we have introduced a parameter $\nu>0$ which measures the strength of the deviation from the reversible dynamics. We remark that the stationary Fokker-Planck equation \eqref{e:fp_stationary} plays precisely the role of the Poisson equation $- \Delta V = \nabla \cdot \bv$ in the Helmholtz decomposition in a flat $L^2$ space. Using now the decomposition \eqref{e:helmholtz_rho} we can decompose the generator $\LL$ of the Markov process $\XX$ on $\T^d$ into a symmetric and an antisymmetric part in $L^2_{\rho_\infty}$ as defined in \eqref{eq:L2space}, representing the reversible and irreversible parts of the dynamics, respectively:
\begin{align}
\LL =\frac{1}{\nu} \cA+ \cS ,
\end{align}
where $\cS = (\nabla \ln \rho_\infty)\cdot \nabla +  \Delta$ and $\cA = \uu \cdot \nabla$. With this abstract setting at hand, the results of \cite{CKRZ08} can  be rephrased in terms of weighted $L^2$ spaces, giving a characterization of relaxation enhancement in terms of eigenfunctions of $\cA$
in the (operator) domain of $\cS^{1/2}$. In this paper, we prove a \emph{quantitive} version of this result in a special case, described in the next section. Estimates on the diffusion coefficient as a function of strength of the nonreversible perturbation, in the abstract setting considered in this subsection, are presented in~\cite{Pavliotis2010}.

Alternatively, one could ask what the effect of a small reversible perturbation of the divergence-free dynamics is on the diffusion coefficient. The two approaches are equivalent, and in this paper, we prefer to consider the influence of a large incompressible flow on the dynamics, as it is clear from the way we write \eqref{eq:bv}. As we shall see, a simple time rescaling makes the problem equivalent to a small compressible perturbation, with noise strength equal to $\sqrt{2\nu}$.

\subsection{Setting and main results}\label{sub:main}
Let $u,v\in C^1(\T)$ be two given functions, and define 
\begin{align}\label{eq:uV}
\uu(x,y)=\begin{pmatrix}
u(y)\\
0
\end{pmatrix},
\qquad V(y)=-\int_0^yv(y')\dd y'.
\end{align}
Throughout the paper, we will assume a zero-mean condition on $v$, namely
\begin{align}\label{eq:meanzerov}
\int_\T v(y)\dd y=0,
\end{align}
and a centering condition for $u$, that is,
\begin{align}\label{eq:meanzerou}
\int_\T u(y)\e^{-V(y)}\dd y=0.
\end{align}
Notice that thanks to \eqref{eq:meanzerov}, $V$ is a periodic potential.
Writing \eqref{eq:sde2} explicitly for the two-component process $\XX(t)=(X(t),Y(t))$, we obtain the system 
\begin{align}\label{eq:SDEsystem0}
\begin{cases}
\displaystyle\dd X(t)=\frac{1}{\nu}u(Y(t))\, dt+\sqrt{2}\,\dd W_1(t),\\
\dd Y(t)= v(Y(t)) \, dt +\sqrt{2}\,\dd W_2(t).
\end{cases}
\end{align}
Note that \eqref{eq:grad2} is automatically satisfied, independently of the choice of $u$ and $v$. 
It turns out that the noise driving the process $\{X(t)\}_{t\geq 0}$ is not essential, so that we consider the stochastic differential equations
\begin{align}\label{eq:SDEsystem}
\begin{cases}
\displaystyle\dd X(t)=\frac{1}{\nu}u(Y(t)) \, dt,\\
\dd Y(t)= v(Y(t)) \, dt +\sqrt{2}\,\dd W(t),
\end{cases}
\end{align}
with initial conditions 
\begin{align}
X(0)=X_0, \qquad Y(0)=Y_0.
\end{align}
Our first main result is a homogenization theorem for a suitable rescaling of the solution of  \eqref{eq:SDEsystem}, with explicit rates of convergence.
\begin{theorem}\label{thm:main1}
Assume $u,v\in C^1(\T)$ are given functions such that \eqref{eq:meanzerov} and \eqref{eq:meanzerou} hold. Consider the solution  $(X(t),Y(t))$ of \eqref{eq:SDEsystem}, and for $\beta>0$ define the one-parameter family of rescaled processes
\begin{align}\label{eq:scalegen}
X^\nu(t)=\nu^{1+\beta} X(t/\nu^{2\beta}), \qquad Y^\nu(t)=\nu^\beta Y(t/\nu^{2\beta}).
\end{align}
Let $\alpha\in \left(0,\frac12\right)$, and $ p\in \left[\frac{1}{2\alpha},\infty\right)$, and assume that the initial conditions satisfy
\begin{align}
\E |X_0|^p<\infty, \qquad \E |Y_0|^p<\infty.
\end{align}
Then, for any $T>0$,
there holds the convergence estimate
\begin{align}
\E \sup_{t\in[0,T]} |X^\nu(t)- \sqrt{2D_u}\, W_x(t)|^p+ \E \sup_{t\in[0,T]} |Y^\nu(t)- \sqrt{2D_v}\, W_y(t)|^p \lesssim \nu^{\alpha \beta p},
\end{align} 
for two independent one-dimensional Brownian motions $W_x(t), \, W_y(t)$, where
\begin{align}
D_u= \|\de_y\chi_u\|^2, \qquad D_v= \|1+\de_y\chi_v\|^2,
\end{align}
and $\chi_u,\chi_v:\T\to \R$ are the unique solutions to the one-dimensional periodic Poisson problems
\begin{align}
&v \de_y \chi_u+\de_{yy} \chi_u=-u, \qquad \int_\T \chi_u(y)\rho_\infty(y)\dd y=0,\\
&v \de_y \chi_v+\de_{yy} \chi_v=-v, \qquad \int_\T \chi_v(y)\rho_\infty(y)\dd y=0.
\end{align}
\end{theorem}

\begin{remark}
It will be clear from the proof that deducing a convergence estimate for the fully diffusive problem \eqref{eq:SDEsystem0} amounts to changing the diffusion coefficient $D_u$ in \eqref{eq:Du} to $\|1+\de_y\chi_u\|^2$.
\end{remark}

Due to the degenerate noise considered for \eqref{eq:SDEsystem}, the unique invariant measure for $(X(t),Y(t))$ is the measure~\cites{HP04, HP08}

\begin{align}\label{eq:Gibbs2}
\rho_\infty(y)= \frac{1}{Z} \e^{- V(y)}, \qquad Z=2\pi \int_{\T} \e^{-V(y)}\dd y,
\end{align}
while \eqref{FP:h} becomes 
\begin{align}\label{FP:h1}
\de_t h+\frac{1}{\nu} u\de_x h=\de_{yy} h-v\de_y h,\qquad h(0,x,y)=h^{in}(x,y).
\end{align}
In what follows, we will be consistent with the notation introduced in \eqref{eq:L2space} and \eqref{eq:L2scalnorm}, so that
\begin{align}
L^2_{\rho_\infty}=\left\{ f:\T^2\to \R, \quad \int_{\T^2} |f(x,y)|^2\rho_\infty(y)\dd x \dd y<\infty, \quad \int_{\T^2}f(x,y)\rho_\infty(y)\dd x\dd y=0 \right\},
\end{align}
and
\begin{align}
\l f,g\r=\int_{\T^2} f(x,y) g(x,y)\rho_\infty(y)\dd x \dd y,\qquad \| f\|^2=\int_{\T^2} |f(x,y)|^2\rho_\infty(y)\dd x \dd y.
\end{align}
An important feature of \eqref{FP:h} is that it decouples in the $x$-Fourier modes. By expanding the solution $h$ as a Fourier series in the  $x$ variable, namely
\begin{align}
h(t,x,y)=\sum_{\ell\in \ZZ} \mathfrak{h}_\ell(t,y)\e^{i\ell x}, \qquad \mathfrak{h}_\ell(t,y)=\frac{1}{2\pi}\int_0^{2\pi}h(t,x,y)\e^{-i\ell x}\dd y.
\end{align} 
for any integer $\ell$ we have from \eqref{FP:h1} that
\begin{align}\label{FP:h1four}
\de_t \mathfrak{h}_\ell+\frac{i\ell}{\nu} u \mathfrak{h}_\ell=\de_{yy} \mathfrak{h}_\ell-v\de_y \mathfrak{h}_\ell,\qquad \mathfrak{h}_\ell(0,y)=\mathfrak{h}_\ell^{in}(y).
\end{align}
However, in order not to deal with  complex-valued function and heavier notation, it is more convenient to deal with functions that 
are localized on a single band $\pm \ell$. Thus, for $k\in\N_0$ we set
\begin{equation}\label{eq:band}
 h_k(t,x,y):=\sum_{|\ell|=k} \mathfrak{h}_\ell(t,y)\e^{i\ell x}.
\end{equation}
This way we may write 
\begin{align}
h(t,x,y)=\sum_{k\in\N_0}h_k(t,x,y),
\end{align} 
as a sum of \emph{real-valued} functions $h_k$ that are localized in $x$-frequency on a single band $\pm k$, $k\in\N_0$. In particular,
for the $x$-average of the function $h$ corresponds to $h_0=\mathfrak{h}_0$.
When norms and scalar products are applied to Fourier modes, it is understood that we will consider the complex one-dimensional version of \eqref{eq:L2scalnorm}, as no confusion will arise.

Our second main result consists of explicit rates of convergence to 0 for $h_k$.

\begin{theorem}\label{thm:main2}
Assume $u,v\in C^2(\T)$ are given functions such that \eqref{eq:meanzerov} and \eqref{eq:meanzerou} hold, and further assume that
$u$ has a finite number of critical points such that $u''(y_{crit})\neq0$.
Then there exist constants $\nu_0,\eps_0\in (0,1)$  such that the following holds:
there exist positive numbers $\aa_0,\bb_0,\cc_0$ only depending on $\eps_0$
for each integer $k\in \N$ and $\nu>0$ with
$\nu k^{-1}\leq \nu_0$
 the energy functional
\begin{align}\label{eq:PHIk}
\Psi_k=\frac12\left[\|h_k\|^2 + \frac{\nu^{1/2}\aa_0}{k^{1/2}} \|\de_y h_k\|^2+\frac{2\bb_0}{k}  \l u' \de_x h_k, \de_y h_k \r+\frac{\cc_0}{\nu^{1/2}k^{3/2}} \| u'\de_x h_k\|^2 \right]
\end{align}
satisfies the differential inequality 
\begin{align}
&\ddt \Psi_k+\eps_0\frac{k^{1/2}}{\nu^{1/2}}\Psi_k+\frac{\aa_0\nu^{1/2}}{2k^{1/2}} \| \de_{yy}h_k-v\de_y h_k\|^2+\frac{\cc_0}{2\nu^{1/2}k^{3/2}}\| u'\de_{xy} h_k\|^2 \leq 0.
\end{align}
for all $t\geq0$.
In particular
\begin{align}
\Psi_k(t) \leq  \e^{-\eps_0\frac{k^{1/2}}{\nu^{1/2}}t}\Psi_k(0), \qquad \forall t\geq 0.
\end{align}
\end{theorem}
The above result confirms that the gradient perturbation given by the potential $V$ does not affect the enhanced diffusion time-scales of the 
backward Kolmogorov equations. From a fluid dynamics perspective, scalars advected by a shear flow $\nu^{-1}(u(y),0)$ (as studied in \cite{BCZ15}) or the slightly compressible perturbation of a shear $\nu^{-1}(u(y),\nu v(y))$ have the same decay properties.

While the dependence on $\nu$ and $k$ of the functional $\Psi_k$ may be cumbersome for interpreting 
the real decay properties of $h_k$, the following simple consequence
entails a clearer result that simply requires that the initial condition be in $L^2_{\rho_\infty}$. The assumptions are the same as in Theorem \ref{thm:main2}.
\begin{corollary}\label{cor:main2}
There exist constants $\nu_0,\eps_0\in (0,1)$ and $c_0>1$ such that the following holds:
for each integer $k\in \N$ and $\nu>0$ with
$\nu k^{-1}\leq \nu_0$
there holds the estimate
\begin{align}\label{eq:L2hk}
\|h_k(t)\|^2 \leq  c_0 \|h^{in}_k\|^2 \e^{-\eps_0\frac{k^{1/2}\nu^{-1/2}}{1+|\ln\nu|+\ln k}t}  , \qquad \forall t\geq 0.
\end{align}
Moreover,
\begin{align}\label{eq:L2h0}
\|h_0(t)\|^2 \leq  \|h^{in}_0\|^2 \e^{-\eps_0 t}  , \qquad \forall t\geq 0.
\end{align}
\end{corollary}
Estimate \eqref{eq:L2hk} contains two very important pieces of information. On the one hand, it quantifies precisely the influence of a large drift 
in the Kolmogorov equation, which allows the convergence mode-by-mode. On the other hand, it shows how the drift has an instantaneous regularization effect in the $x$-variable, in which diffusion is not present: from $L^2$ initial data, \eqref{eq:PHIk} shows that Fourier coefficients decay exponentially fast, giving rise to Gevrey-type regularization effects. It is worth mentioning that, due to the results in \cite{CZDri19}, the  decay rate is 
optimal up to the logarithmic correction.

Since the $x$-average of $h$ is not influenced by the drift (see \eqref{FP:h1four} for $k=0$),  the results of Corollary \ref{cor:main2} can be stated
for the solution of the advection diffusion equation \eqref{FP:h1} as follows. 
\begin{corollary}\label{cor:main2real}
There exist constants $\nu_0\in (0,1)$ and $c_0>1$ such that the following holds:
for each $\nu\in(0,\nu_0]$ with
there holds the estimate
\begin{equation}\label{eq:L2hreal}
\left\|h(t)-\int_\T h(t,x,y)\dd x\right\|^2 \leq  c_0 \left\|h^{in}-\int_\T h^{in}(x,y)\dd x\right\|^2 \e^{-\eps_0\frac{\nu^{-1/2}}{1+|\ln\nu|}t}  , \qquad \forall t\geq 0,
\end{equation}
for any $h^{in}\in L^2_{\rho_\infty}$.
\end{corollary}

The result of Corollary \ref{cor:main2real} is easily explained by looking at Figure \ref{fig:shear}, in which various snapshots of the solution 
of \eqref{FP:h1} with $u(y)=-3\cos (3y)$ and $v(y)\equiv 0$ are plotted. 

\begin{figure}[h!]
  \centering
  \begin{subfigure}[b]{0.24\linewidth}
    \includegraphics[width=\linewidth]{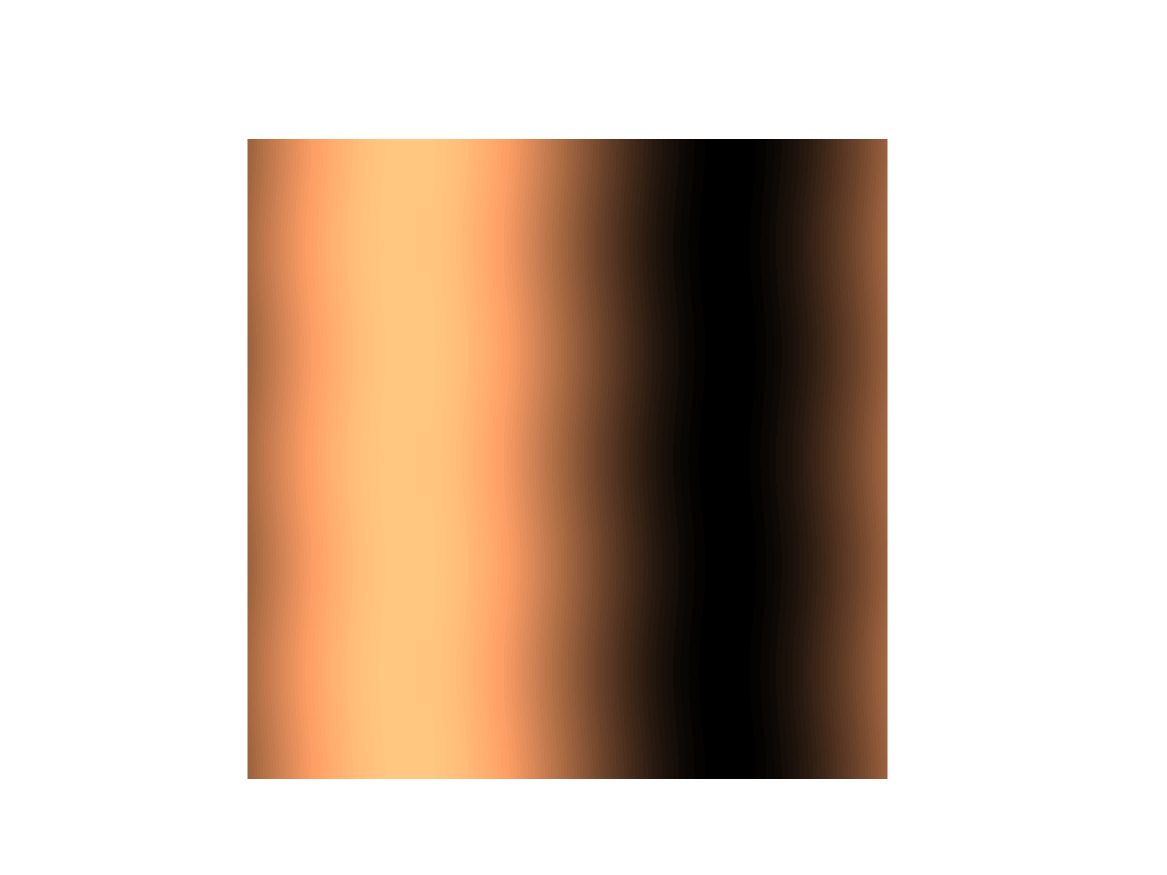}
  \end{subfigure}
  \begin{subfigure}[b]{0.24\linewidth}
    \includegraphics[width=\linewidth]{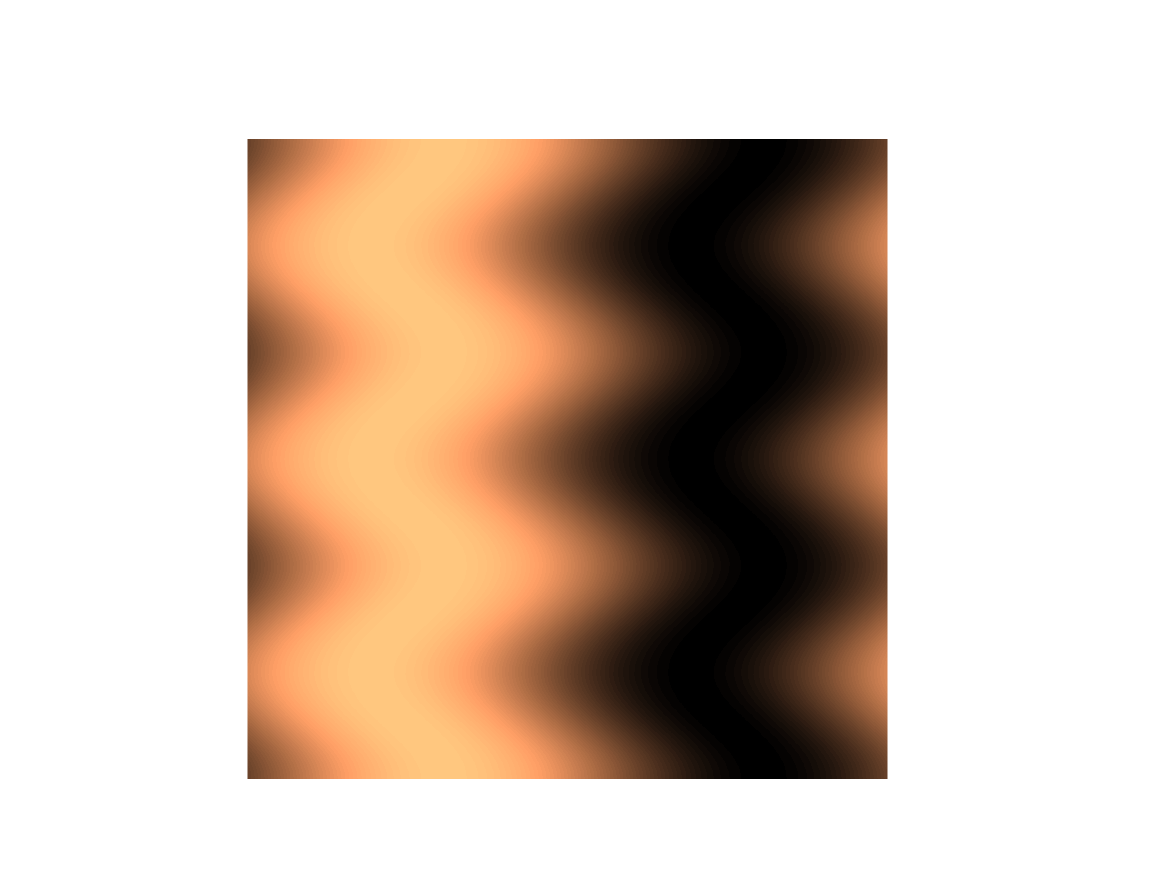}
  \end{subfigure}
  \begin{subfigure}[b]{0.24\linewidth}
    \includegraphics[width=\linewidth]{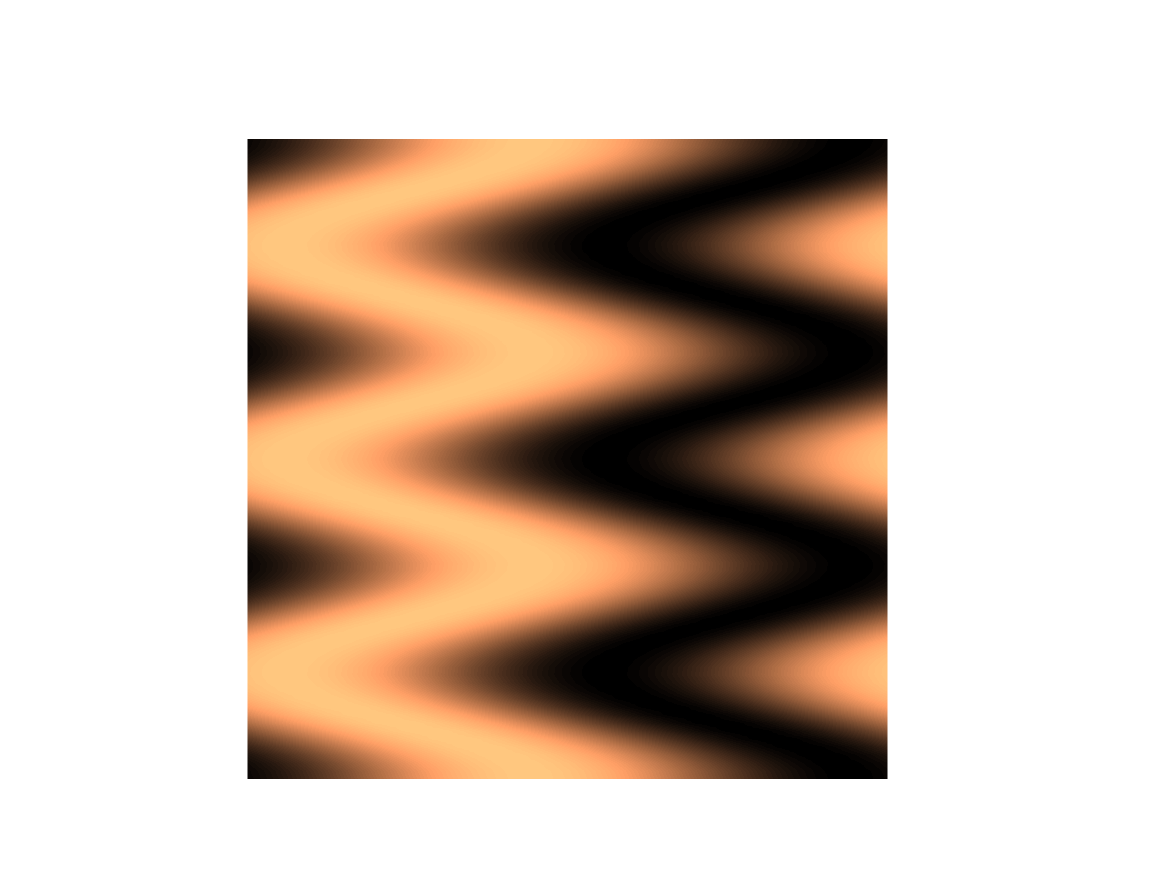}
  \end{subfigure}
  \begin{subfigure}[b]{0.24\linewidth}
    \includegraphics[width=\linewidth]{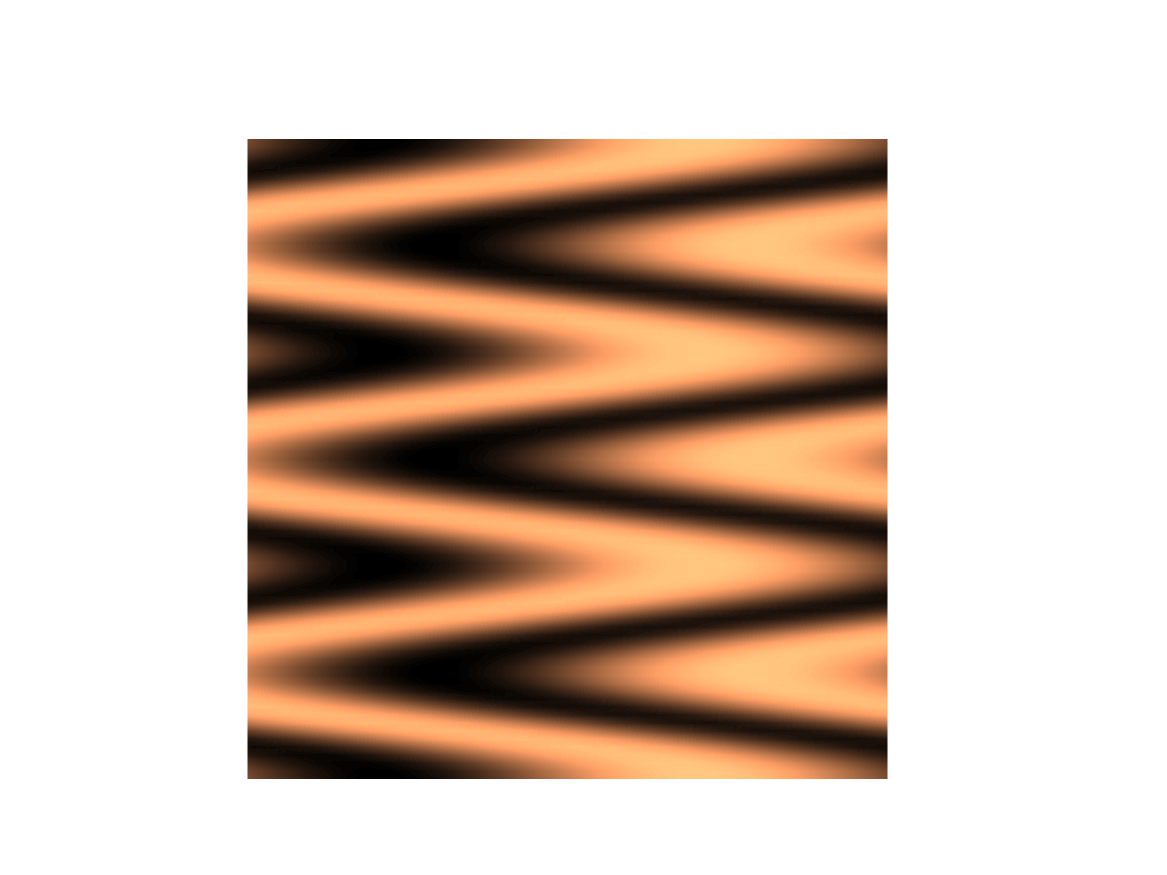}
  \end{subfigure}
    \begin{subfigure}[b]{0.24\linewidth}
    \includegraphics[width=\linewidth]{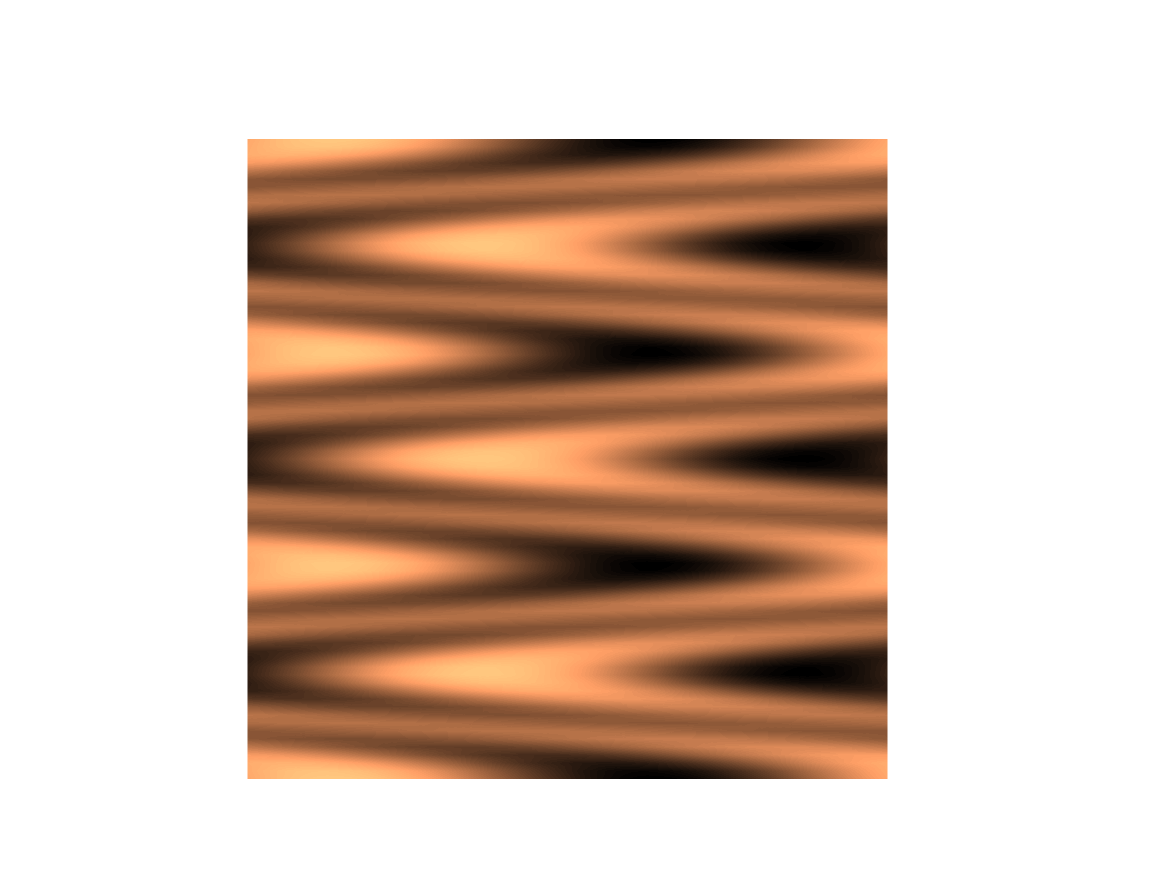}
  \end{subfigure}
  \begin{subfigure}[b]{0.24\linewidth}
    \includegraphics[width=\linewidth]{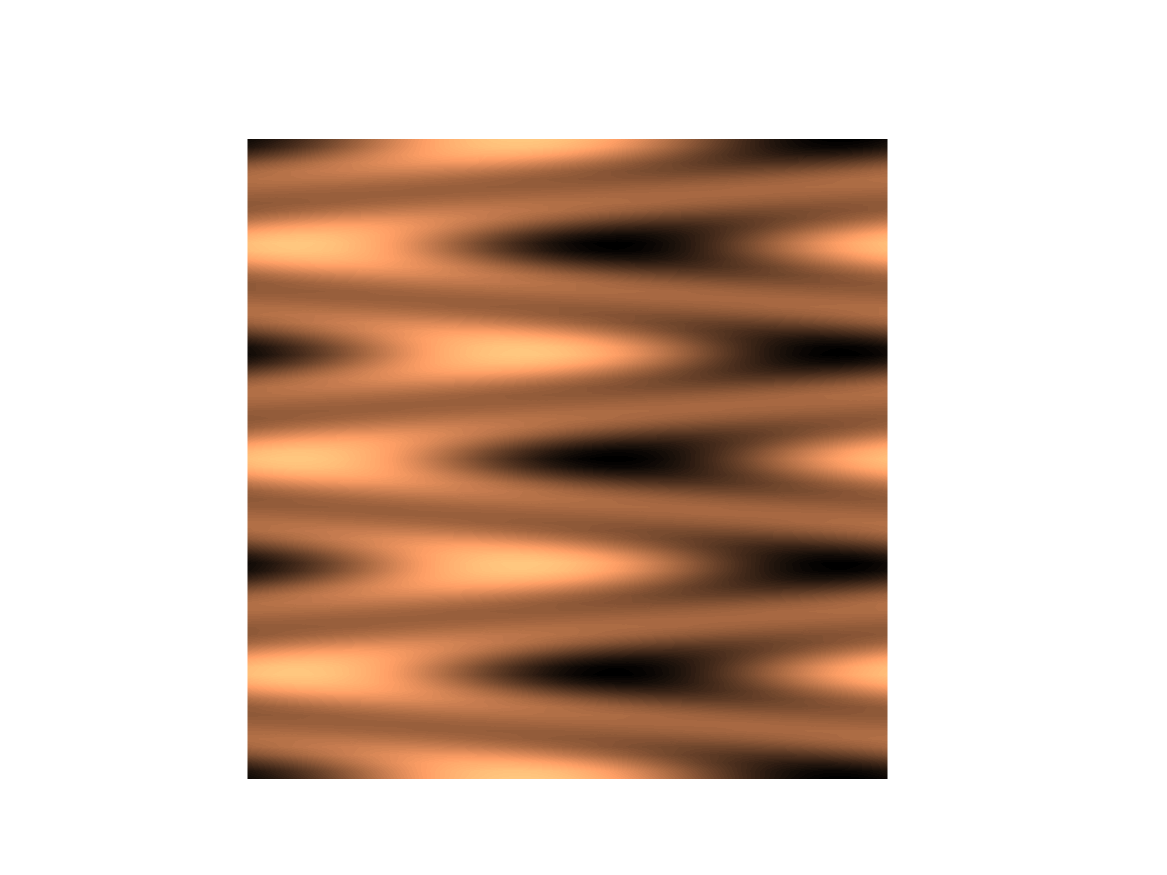}
  \end{subfigure}
  \begin{subfigure}[b]{0.24\linewidth}
    \includegraphics[width=\linewidth]{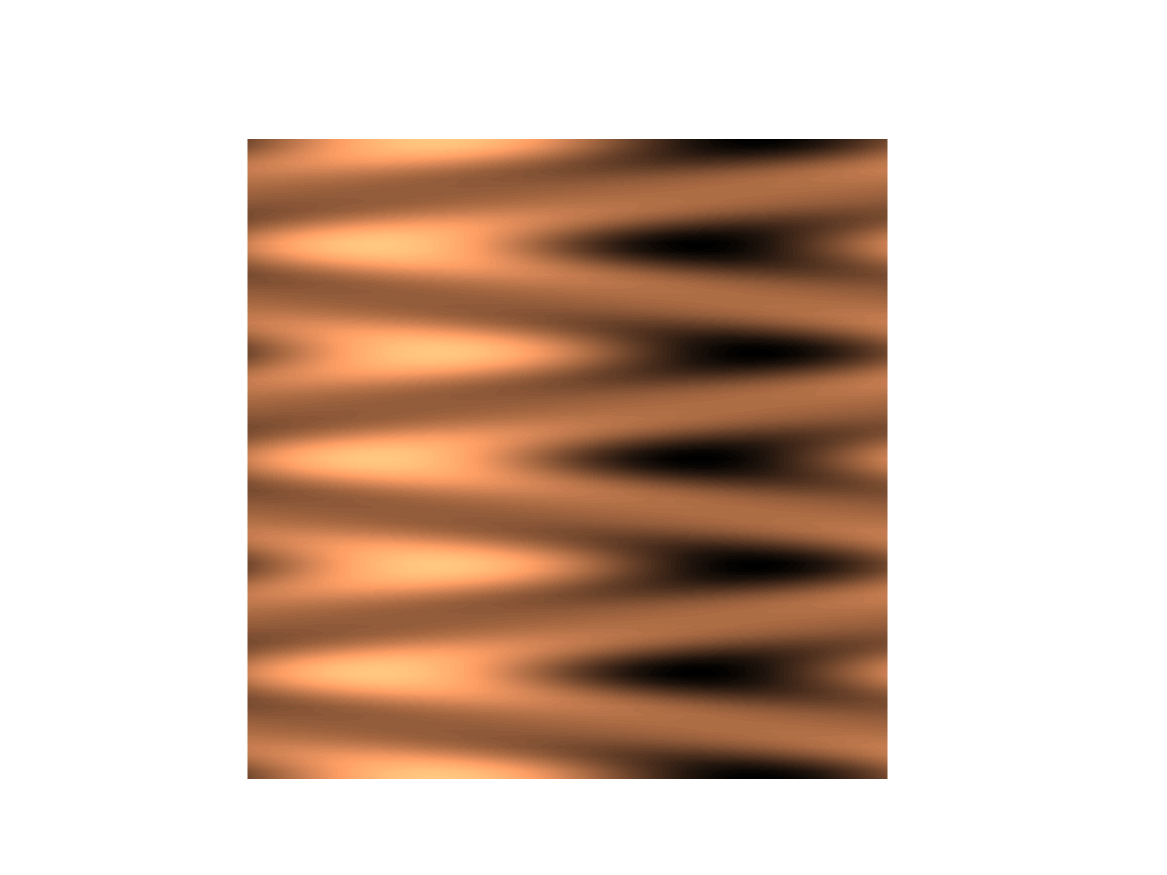}
  \end{subfigure}
  \begin{subfigure}[b]{0.24\linewidth}
    \includegraphics[width=\linewidth]{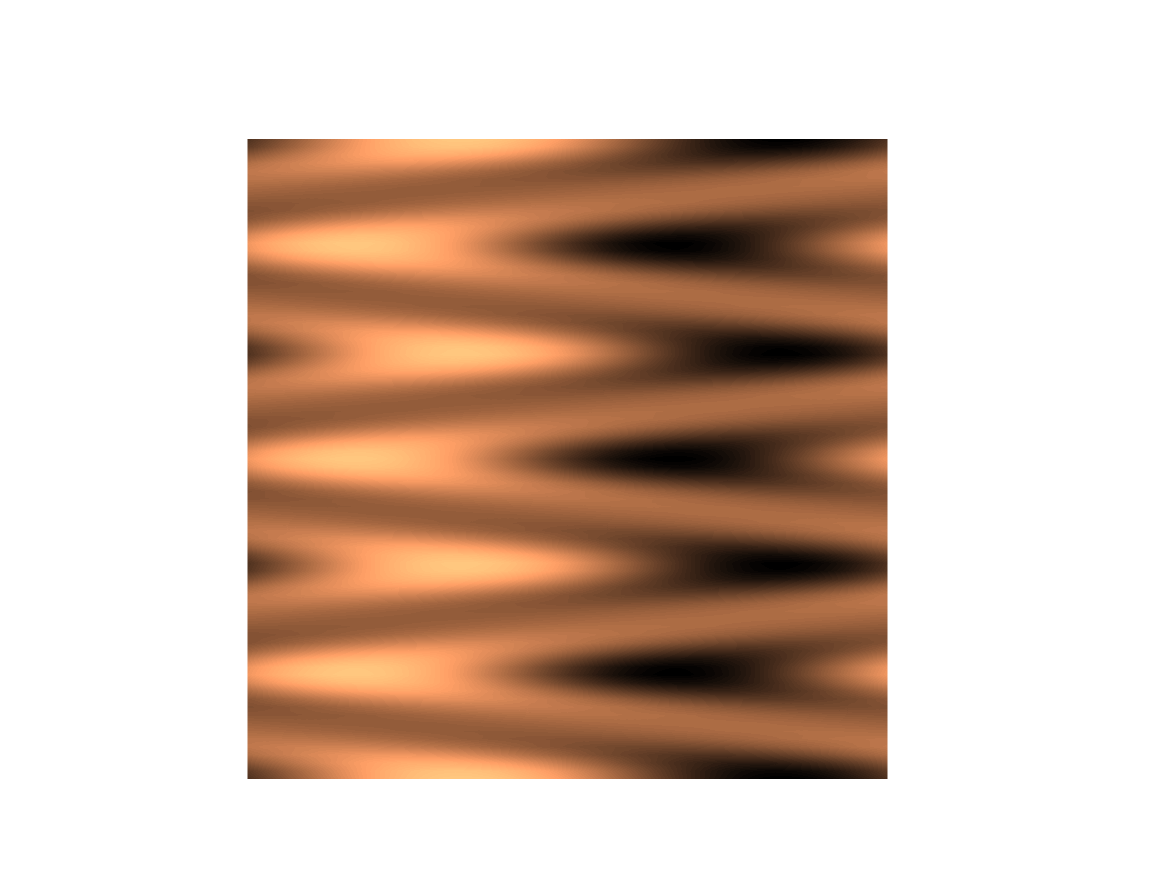}
  \end{subfigure}
  \caption{The evolution of the solution $h$ to \eqref{FP:h1} with $u(y)=-3\cos (3y)$, $v(y)\equiv 0$, $h^{in}(x,y)=\sin x$ and $\nu=10^{-3}$. The simulation has been performed using FreeFem++, with finite elements P1 for the space discretization.}
  \label{fig:shear}
\end{figure}

For time-scales faster than $O(\nu^{1/2}(1+|\ln\nu|))$,  the dominant behavior is mixing by incompressible velocities: \eqref{eq:L2hreal} says that if $\nu$ is chosen small enough, $h$ is very close to its $x$-average, and hence tends to become $x$-independent. This coincides with the appearance of ``horizontal'' stripes. At this point, only diffusion is relevant, and the solution slowly decays to zero at a $\nu$-independent rate, as prescribed by \eqref{eq:L2h0}.

\section{Rates of convergence in homogenization}

In this section we  prove Theorem \ref{thm:main1}. The main ingredient to prove explicit convergence rate is the use
of a second Poisson equation and an auxiliary process that we introduce in the next section. The rates follows from the
Dambis-Dubins-Schwarz theorem (see \cite{KSh91}*{Thm 3.4.6}) on time-change for martingales, together with the
H\"older continuity properties of Brownian motion.

\subsection{Auxiliary processes and the Poisson equation}
In order to prove Theorem \ref{thm:main1}, we derive from \eqref{eq:SDEsystem} the SDEs for the rescaled process
in \eqref{eq:scalegen}, which read
\begin{equation}\label{eq:SDErescaled1}
\begin{aligned}
\dd X^\nu(t) &=\frac{1}{\nu^\beta} u\left(\frac{Y^\nu(t)}{\nu^\beta}\right)\dd t,\\
\dd Y^\nu(t) &=\frac{1}{\nu^\beta} v\left(\frac{Y^\nu(t)}{\nu^\beta}\right)\dd t+\sqrt{2}\, \dd W(t),
\end{aligned}
\end{equation}
with initial conditions
\begin{align}\label{eq:INrescaled}
X^\nu(0)=\nu^{1+\beta} X_0, \qquad Y^\nu(0)=\nu^{\beta} Y_0.
\end{align}
In fact, it is convenient to introduce the auxiliary process 
\begin{align}
R^\nu(t)=\frac{Y^\nu(t)}{\nu^\beta},
\end{align}
and re-write \eqref{eq:SDErescaled1} as an augmented system of the form
\begin{align}
\dd X^\nu(t) &=\frac{1}{\nu^\beta} u\left(R^\nu(t)\right)\dd t,\label{eq:SDErescaledX}\\
\dd Y^\nu(t) &=\frac{1}{\nu^\beta} v\left(R^\nu(t)\right)\dd t+\sqrt{2}\, \dd W(t),\label{eq:SDErescaledY}\\
\dd R^\nu(t) &=\frac{1}{\nu^{2\beta}} v\left(R^\nu(t)\right)\dd t+\frac{\sqrt{2}}{\nu^\beta}\, \dd W(t)\label{eq:SDErescaledZ}.
\end{align}
In what follows, an important role will be played by the generator $\LL$  of the process $R^1$ (i.e. for $\nu=1$), namely the operator
\begin{align}
\LL =v \de_y +\de_{yy}.
\end{align}
Given a function $\phi\in C^1(\T)$ such that
\begin{align}
\int_\T \phi(y)\rho_\infty(y)\dd y=0,
\end{align}
it is not hard to verify that the unique solution to the Poisson equation
\begin{align}\label{eq:STURM}
\LL \chi =\phi, \qquad \int_\T \chi(y)\rho_\infty(y)\dd y=0,
\end{align}
is given by
\begin{align}
\chi(y)=B+\int_0^y\e^{V(y')}\left[A+ \int_0^{y'}\phi(\bar{y})\e^{-V(\bar{y})} \dd \bar{y}\right]\dd y',
\end{align}
where the constant $A,B$ are chosen to enforce periodicity and the weighted mean-zero condition in \eqref{eq:STURM} as
\begin{align}
A=-\left[\int_\T \e^{V(y')}\dd y'\right]^{-1}\int_\T\int_0^{y'}\phi(\bar{y})\e^{V(y')-V(\bar{y})} \dd \bar{y}\,\dd y'
\end{align}
and
\begin{align}
B=-\int_\T\rho_\infty(y)\int_0^y\e^{V(y')}\left[A+ \int_0^{y'}\phi(\bar{y})\e^{-V(\bar{y})} \dd \bar{y}\right]\dd y' \dd y.
\end{align}
Here, the notation is that of \eqref{eq:uV} and \eqref{eq:Gibbs2}. In particular, notice that $\chi\in C^1(\T)$ (at least).

\subsection{Convergence rates for $Y^\nu$}

We begin to deal with the convergence for the process $Y^\nu$.

\begin{lemma}\label{lem:Yconv}
Let $\alpha\in \left(0,\frac12\right)$ and $ p\in \left[\frac{1}{2\alpha},\infty\right)$ be arbitrarily fixed, and assume that
\begin{align}\label{eq:initY}
\E |Y_0|^p<\infty.
\end{align}
Then we have the convergence estimate
\begin{align}
\E \sup_{t\in[0,T]} |Y^\nu(t)- \sqrt{2D_v}\, W(t)|^p \lesssim \nu^{\alpha \beta p},
\end{align} 
where
\begin{align}
D_v= \|1+\de_y\chi_v\|^2=\int_\T \left[1+\de_y\chi_v(y)\right]^2 \rho_\infty(y)\dd y
\end{align}
and $\chi_v:\T\to \R$ is the unique solution to the elliptic equation
\begin{align}\label{eq:Poissonv}
\LL\chi_v=-v, \qquad \int_\T \chi_v(y)\rho_\infty(y)\dd y=0.
\end{align}
\end{lemma}

\begin{proof}
First of all, notice that in light of \eqref{eq:uV} and \eqref{eq:Gibbs2}, we have
\begin{align}
\int_\T v(y)\rho_\infty(y)\dd y =-\frac{1}{Z}\int_\T V'(y) \e^{- V(y)}\dd y=\frac{1}{Z}\int_\T \de_y\left(\e^{- V(y)}\right)\dd y=0.
\end{align}
As a consequence, the unique solution to \eqref{eq:Poissonv} satisfies $\chi_v\in C^1(\T)$.
Then, using Ito's formula we find
\begin{align}
\dd \chi_v (R^\nu)&=\frac{1}{\nu^{2\beta}} \LL \chi_v\left(R^\nu\right)+\frac{\sqrt{2}}{\nu^{\beta}} \de_y\chi_v(R^\nu) \dd W
=-\frac{1}{\nu^{2\beta}}  v\left(R^\nu\right) \dd t +\frac{\sqrt{2}}{\nu^{\beta}} \de_y\chi_v(R^\nu) \dd W.
\end{align}
In turn,
\begin{align}
\frac{1}{\nu^{\beta}}\int_0^t  v\left(R^\nu(s)\right) \dd s = -\nu^{\beta} \left[\chi_v (R^\nu(t)) - \chi_v (R^\nu(0))\right]+\sqrt{2} \int_0^t \de_y\chi_v(R^\nu(s)) \dd W(s).
\end{align}
Substituting in the equation for $Y^\nu$, we arrive at
\begin{align}
Y^\nu(t)=Y^\nu(0)   -\nu^{\beta} \left[\chi_v (R^\nu(t)) - \chi_v (R^\nu(0))\right]+\sqrt{2} \int_0^t \left[1+\de_y\chi_v(R^\nu(s)) \right]\dd W(s).
\end{align}
Thus
\begin{align}
Y^\nu(t)- \sqrt{2D_v}\, W(t)&=Y^\nu(0)    -\nu^{\beta} \left[\chi_v (R^\nu(t)) - \chi_v (R^\nu(0))\right]\notag\\
&\quad+\sqrt{2} \int_0^t \left[1+\de_y\chi_v(R^\nu(s)) -\sqrt{D_v}\right]\dd W(s)
\end{align}
Hence, for any $p\geq 1$ we have
\begin{align}\label{eq:Yeps}
|Y^\nu(t)- \sqrt{2D_v}\, W(t)|^p &\lesssim |Y^\nu(0)|^p  + \nu^{ p\beta} \left[|\chi_v (R^\nu(t))|^p +|\chi_v (R^\nu(0))|^p\right]\notag\\
&\quad+ \left|\int_0^t \left(1+\de_y\chi_v(R^\nu(s)) -\sqrt{D_v}\right)\dd W(s)\right|^p.
\end{align} 
The first two terms are essentially harmless, given the fact that $\chi_v$ is at least $C^1$, the assumption \eqref{eq:initY} on the initial condition and the rescaling  \eqref{eq:INrescaled} . To estimate the last term,  first note that  the Dambis-Dubins-Schwarz theorem (see \cite{KSh91}*{Thm 3.4.6})
implies that in law we have
\begin{align}
\int_0^t \left(1+\de_y\chi_v(R^\nu(s)) -\sqrt{D_v}\right)\dd W(s) = W\left( \int_0^t \left(1+\de_y\chi_v(R^\nu(s)) \right)^2\dd s \right) -W(D_vt).
\end{align}
Hence, by the H\"older-continuity of Brownian motion and the Cauchy-Schwarz inequality inequality we deduce that
\begin{align}\label{eq:compafsed}
&\E \sup_{t\in[0,T]} \left|\int_0^t \left(1+\de_y\chi_v(R^\nu(s)) -\sqrt{D_v}\right)\dd W(s)\right|^p\notag\\
&\qquad\qquad=\E \sup_{t\in[0,T]} \left|W\left( \int_0^t \left(1+\de_y\chi_v(R^\nu(s)) \right)^2\dd s \right) -W(D_vt)\right|^p \notag\\
&\qquad\qquad\lesssim \E \left[\mathrm{Hol}^p_\alpha (W(t)) \sup_{t\in[0,T]}  \left|\left( \int_0^t \left(1+\de_y\chi_v(R^\nu(s)) \right)^2\dd s \right) -D_vt\right|^{\alpha p}  \right] \notag\\
&\qquad\qquad\lesssim    \left(\E\left[ \sup_{t\in[0,T]}  \left|\left( \int_0^t \left(1+\de_y\chi_v(R^\nu(s)) \right)^2\dd s \right) -D_vt\right|^{2\alpha p}  \right]\right)^{1/2} \notag\\
&\qquad\qquad\lesssim    \left(\E\left[ \sup_{t\in[0,T]}  \left|\left( \int_0^t g(R^\nu(s))\dd s \right)\right|^{2\alpha p}  \right]\right)^{1/2},
\end{align} 
where $\mathrm{Hol}_\alpha (W(t))$ is the H\"older constant of the Brownian motion $W(t)$ and we have conveniently defined
\begin{align}
g(y)=\left(1+\de_y\chi_v(y) \right)^2-D_v.
\end{align}
In this way,
\begin{align}
\int_\T g(y)\rho_\infty(y)\dd y=0.
\end{align}
Let now $\chi_g\in C^1(\T)$ be the unique solution to
\begin{align}
\LL \chi_g=-g, \qquad \int_\T \chi_g(y)\rho_\infty(y)\dd y=0.
\end{align}
As before, using Ito's formula we find
\begin{align}
\dd \chi_g (R^\nu)=-\frac{1}{\nu^{2\beta}}  g\left(R^\nu\right) \dd t +\frac{\sqrt{2}}{\nu^{\beta}} \de_y\chi_g(R^\nu) \dd W.
\end{align}
In turn,
\begin{align}
\int_0^t  g\left(R^\nu(s)\right) \dd s = -\nu^{2\beta} \left[\chi_g  (R^\nu(t)) - \chi_g  (R^\nu(0))\right]+\nu^\beta\sqrt{2}\int_0^t \de_y\chi_g (R^\nu(s)) \dd W(s).
\end{align}
Using the Burkholder-Davis-Gundy inequality, we then find for any $q\geq 1$ that
\begin{align}
\E \sup_{t\in [0,T]}\left|\int_0^t  g\left(R^\nu(s)\right) \dd s \right|^{q} 
&\lesssim \nu^{2q\beta}\,\E \sup_{t\in [0,T]} \left[ |\chi_g (R^\nu(t))|^{q} + |\chi_g (R^\nu(0))|^{q}\right]\notag\\
&\quad+\nu^{q\beta}\E \sup_{t\in [0,T]}\left|\int_0^t \de_y\chi_g(R^\nu(s)) \dd W(s)\right|^q\notag\\
&\lesssim \nu^{2q\beta}+\nu^{q\beta} \int_0^T \E\left|\de_y\chi_g(R^\nu(s)) \right|^q\dd s\lesssim \nu^{q\beta}.
\end{align}
Going back to \eqref{eq:compafsed} and using the above bound (here is where we need $2\alpha p\geq 1$),  we find
\begin{align}
\E \sup_{t\in[0,T]} \left|\int_0^t \left(1+\de_y\chi_v(R^\nu(s)) -\sqrt{D_v}\right)\dd W(s)\right|^p
\lesssim    \nu^{p\alpha\beta}
\end{align} 
From \eqref{eq:Yeps} and using the rescaling of the initial data \eqref{eq:INrescaled}, we take the supremum in time and expectations to deduce the desired estimate
\begin{align}
\E \sup_{t\in[0,T]} |Y^\nu(t)- \sqrt{2D_v}\, W(t)|^p
& \lesssim  \E |Y^\nu(0)|^p+ \nu^{p\beta} \E \sup_{t\in[0,T]} \left[|\chi_v (R^\nu(t))|^p +|\chi_v (R^\nu(t))|^p\right]+\nu^{ p\alpha\beta}\notag\\
&\lesssim \nu^{p\beta}\E |Y_0|^p+\nu^{ p\alpha\beta}.
\end{align} 
This concludes the proof.
\end{proof}

\subsection{Convergence rates for $X^\nu$}
We now turn to the process $X^\nu$. The proof is somewhat similar, so we will only highlight the main points.
\begin{lemma}\label{lem:Xconv}
Let $\alpha\in \left(0,\frac12\right)$ and $ p\in \left[\frac{1}{2\alpha},\infty\right)$ be arbitrarily fixed, and assume that
\begin{align}
\E |X_0|^p<\infty.
\end{align}
Then we have the convergence estimate
\begin{align}
\E \sup_{t\in[0,T]} |X^\nu(t)- \sqrt{2D_u}\, W(t)|^p \lesssim \nu^{\alpha \beta p},
\end{align} 
where
\begin{align}\label{eq:Du}
D_u= \|\de_y\chi_u\|^2
\end{align}
and $\chi_u:\T\to \R$ is the unique solution to the elliptic equation
\begin{align}\label{eq:Poissonu}
\LL\chi_u=-u, \qquad \int_\T \chi_u(y)\rho_\infty(y)\dd y=0.
\end{align}
\end{lemma}

\begin{proof}
From \eqref{eq:INrescaled} and \eqref{eq:SDErescaledX}, we have that
\begin{align}
X^\nu(0)=\nu^{1+\beta} X_0 +\frac{1}{\nu^{\beta}}\int_0^{t} u \left(R^\nu(s)\right) \dd s.
\end{align}
Therefore, using assumption \eqref{eq:meanzerou}, \eqref{eq:Poissonu}  and the Ito's formula we infer that
\begin{align}
\dd \chi_u (R^\nu)=-\frac{1}{\nu^{2\beta}}  u\left(R^\nu\right) \dd t +\frac{\sqrt{2}}{\nu^{\beta}} \de_y\chi_u(R^\nu) \dd W.
\end{align}
As before,
\begin{align}
\frac{1}{\nu^{\beta}}\int_0^t  u\left(R^\nu(s)\right) \dd s = -\nu^\beta \left[\chi_u (R^\nu(t)) - \chi_u (R^\nu(0))\right]+ \sqrt{2}  \int_0^t \de_y\chi_u(R^\nu(s)) \dd W_s,
\end{align}
so that
\begin{align}
X^\nu(t)-\sqrt{2D_u}\, W(t)=&\nu^{1+\beta} X_0 -\nu^\beta \left[\chi_u (R^\nu(t)) - \chi_u (R^\nu(0))\right]\notag\\
&+ \sqrt{2}  \int_0^t \left[\de_y\chi_u(R^\nu(s))-\sqrt{D_u} \right]\dd W(s).
\end{align}
By repeating the same proof as in Lemma \ref{lem:Yconv}, we arrive at
\begin{align}
\E \sup_{t\in[0,T]} \left|\int_0^t \left(\de_y\chi_u(R^\nu(s)) -\sqrt{D_u}\right)\dd W(s)\right|^p
\lesssim    \nu^{p\alpha\beta}.
\end{align} 
Hence, the convergence estimate follows immediately from
\begin{align}
\E \sup_{t\in[0,T]} |X^\nu(t)-\sqrt{2D_u}\, W(t)|^p \lesssim  \nu^{p(1+\beta)}\E |X_0|^p  + \nu^{ p\beta} \E \sup_{t\in[0,T]} \left[|\chi_u (R^\nu(t))|^p +|\chi_u (R^\nu(0))|^p\right]+ \nu^{p\alpha \beta },
\end{align} 
and the proof is over.
\end{proof}
It is clear that Theorem \ref{thm:main1} is simply a combination of Lemma \ref{lem:Yconv} and Lemma \ref{lem:Xconv}. 

%
%
\section{Enhanced diffusion with slightly compressible perturbations}
This section is devoted to the proof of Theorem \ref{thm:main2}. Instead of working with \eqref{FP:h1}, we rescale time by defining
\begin{align}\label{eq:time-scaling}
f(t,x,y)=h(\nu t,x,y),
\end{align}
so that from \eqref{FP:h1} we infer that $f$ satisfies
\begin{align}\label{FP:f1}
\de_t f+ u\de_x f=\nu\left(\de_{yy}f-v\de_y f\right),\qquad f(0,x,y)=f^{in}(x,y)=h^{in}(x,y).
\end{align}
Clearly, the  Fourier-decomposition  in \eqref{eq:band} is preserved, and therefore we can write the above as
\begin{align}\label{FP:f1four}
\de_t f_k+ u\de_x f_k=\nu\left(\de_{yy}f_k-v\de_y f_k\right),\qquad f_k(0,y)=f_k^{in}(y),
\end{align}
having in mind the localization to band $k\in\N$.
The analogous of Theorem \ref{thm:main2} for $f$ is the following.

\begin{theorem}\label{thm:main3}
There exist constants $\nu_0,\eps_0\in (0,1)$  such that the following holds:
there exist positive numbers $\aa_0,\bb_0,\cc_0$ only depending on $\eps_0$
such that for each integer $k\neq 0$ and $\nu>0$ with
$\nu k^{-1}\leq \nu_0$
 the energy functional
\begin{align}\label{eq:PHIkk}
\Phi_k=\frac12\left[\|f_k\|^2 + \frac{\nu^{1/2}\aa_0}{k^{1/2}} \|\de_y f_k\|^2+\frac{2\bb_0}{k}  \l  u' \de_x f_k, \de_y f_k \r+\frac{\cc_0}{\nu^{1/2}k^{3/2}} \|  u' \de_x f_k\|^2 \right]
\end{align}
satisfies the differential inequality 
\begin{align}\label{eq:diffPHI2}
&\ddt \Phi_k+\eps_0\nu^{1/2}k^{1/2}\Phi_k+\frac{\aa_0\nu^{1/2}}{2k^{1/2}} \| \de_{yy}f_k-v\de_y f_k\|^2+\frac{\cc_0}{2\nu^{1/2}k^{3/2}}\|  u'\de_{xy} f_k\|^2 \leq 0.
\end{align}
for all $t\geq0$.
In particular
\begin{align}
\Phi_k(t) \leq  \e^{-\eps_0\nu^{1/2}k^{1/2}t}\Phi_k(0), \qquad \forall t\geq 0.
\end{align}
\end{theorem}
It is clear that Theorem \ref{thm:main2} is a straightforward consequence of the above result, provided we rescale time according to \eqref{eq:time-scaling}. Notice that finiteness of $\Phi_k(0)$  requires $f^{in}_k,\de_yf^{in}_k\in L^2_{\rho_\infty}$ for each $k\in \ZZ$. 
At the cost of a logarithmic  loss on the rate, it is possible to relax this requirement. 

\begin{corollary}\label{cor:main3}
There exist constants $\nu_0,\eps_0\in (0,1)$ and $c_0>1$ such that the following holds:
for each integer $k\neq 0$ and $\nu>0$ with
$\nu k^{-1}\leq \nu_0$
there holds the estimate
\begin{align}
\|f_k(t)\|^2 \leq  c_0 \|f^{in}_k\|^2 \e^{-\eps_0\frac{\nu^{1/2}k^{1/2}}{1+|\ln\nu|+\ln k}t}  , \qquad \forall t\geq 0.
\end{align}
\end{corollary}
The above result is a semigroup estimate for the solution operator of \eqref{FP:f1four}, and it implies the result of Corollary \ref{cor:main2} for $k\neq 0$. The proof of this is postponed in Section \ref{sub:semigroup}.

\begin{remark}[The $k=0$ mode]
The $k=0$ mode (or the $x$-average of $f$ in real variables)  satisfies the equation
\begin{align}\label{eq:f0}
\de_t f_0= \nu(\de_{yy}f_0 -v\de_y f_0).
\end{align}
In view of the fact that
\begin{align}
\int_{\T^2} f(t,x,y)\rho_\infty(y)=0, \qquad \forall t\geq0, 
\end{align}
we have that 
\begin{align}\label{eq:meanzerof0}
\int_{\T} f_0(t,y)\rho_\infty(y)=0, \qquad \forall t\geq0.
\end{align}
Note that the above is only true for $f_0$, and it cannot in general be imposed on any other Fourier modes.
A simple energy estimate performed on \eqref{eq:f0} implies that
\begin{align}
\frac12\ddt \| f_0\|^2+\nu\|\de_yf_0\|^2=0,
\end{align}
so, in view of \eqref{eq:meanzerof0}, we are in the position of applying the Poincar\'e inequality and obtain
\begin{align}
\| f_0(t)\|\leq \| f^{in}_0\| \e^{-\eps_0\nu t}, \qquad \forall t\geq 0,
\end{align}
for some $\eps_0>0$. In turn,
\begin{align}
\| h_0(t)\|\leq \| h^{in}_0\| \e^{-\eps_0 t}, \qquad \forall t\geq 0,
\end{align}
as stated in  Corollary \ref{cor:main2}. In other words, the $x$-average of $h$ does not see the effect of the drift $u$.
\end{remark}

\subsection{Some energy estimates}
In this section, we perform energy estimates on \eqref{FP:f1four}.  In what follows,
we will tacitly make use of the antisymmetry properties
\begin{align}
\l u\de_x g, g\r=-\l g, u\de_x g\r, \qquad \l u\de_x g,g\r=0.
\end{align}
Also, we define the operator
\begin{align}
L= \de_{yy} -v\de_y 
\end{align}
which satisfies the symmetry properties
\begin{align}
\l L g, g\r=-\l \de_y g, \de_y g\r, \qquad \l L g, g\r=-\| \de_y g\|^2.
\end{align}
Testing \eqref{FP:f1} with $f$ in $L^2_{\rho_\infty}$ we have 
\begin{align}\label{eq:L2}
 \frac12\ddt \|f\|^2  =-\nu\| \de_y f\|^2.
\end{align}
Analogously, taking $\de_y$ of \eqref{FP:f1} and testing with $\de_y f$ we obtain
\begin{align}\label{eq:alpha}
\frac12\ddt \|\de_y f\|^2&=\nu\l\de_y Lf, \de_y f  \r -\l \de_y (u\de_x f),\de_y f \r= -\nu\|L f\|^2 -\l u'\de_x f,\de_y f \r.
\end{align}
We now turn to the cross term $\l u'\de_x f,\de_y f \r$. Using  \eqref{FP:f1}, we have
\begin{align}
\ddt\l u'\de_x f,\de_y f \r=\nu\left[\l u'\de_x Lf,\de_y  f \r+\l u'\de_x f,\de_y Lf \r\right]-\left[\l u'\de_x (u\de_x f),\de_y f \r+\l u'\de_x f,\de_y (u\de_x f)\r\right].
\end{align}
Now, recalling that $\de_y \rho_\infty=v \rho_\infty$, we have
\begin{align}
\l u'\de_x f,\de_y Lf \r=\int_{\T^2}u'\de_x f \de_y Lf \rho_\infty \dd x \dd y=-\l u'\de_{xy} f,Lf \r-\l u''\de_x f, Lf \r-\l v u'\de_x f, Lf \r.
\end{align}
and therefore
\begin{align}
\l u'\de_x Lf,\de_y  f \r+\l u'\de_x f,\de_y Lf \r=-2\l u'\de_{xy} f,Lf \r-\l u''\de_x f, Lf \r-\l v u'\de_x f, Lf \r.
\end{align}
On the other hand,
\begin{align}
\l u'\de_x (u\de_x f),\de_y f \r+\l u'\de_x f,\de_y (u\de_x f)\r=\| u'\de_x f\|^2,
\end{align}
 and therefore
\begin{align}\label{eq:beta}
\ddt\l u'\de_x f,\de_y f \r=-\nu\left[2\l u'\de_{xy} f,Lf \r+\l u''\de_x f, Lf \r+\l v u'\de_x f, Lf \r\right]-\| u'\de_x f\|^2.
\end{align}
We are left with one more energy estimate, namely
\begin{align}
\frac12\ddt\| u'\de_x f\|^2&= \nu \l u'\de_x f,u'\de_x Lf \r -\l u'\de_x f,u'\de_x (u\de_x f) \r =- \nu \l (u')^2\de_{xx} f, Lf \r \notag\\
&= \nu \l \de_y[(u')^2\de_{xx} f], \de_yf \r= \nu \l (u')^2\de_{xxy} f, \de_yf \r +2\nu \l u' u'' \de_{xx} f, \de_yf \r \notag\\
&=- \nu \| u'\de_{xy} f\|^2 +2\nu \l u' u'' \de_{xx} f, \de_yf \r,
\end{align}
which we rewrite for further reference as
\begin{align}\label{eq:gamma}
\frac12\ddt\| u'\de_x f\|^2=- \nu \| u'\de_{xy} f\|^2 +2\nu \l  u'' \de_{x} f, u'\de_{xy}f \r.
\end{align}
We now combine the above estimates in a precise way in order to derive a good differential inequality.
%
%
\subsection{The hypocoercivity scheme}
For $\aa,\bb,\cc>0$ to be determined, define the functional
\begin{align}
\Phi=\frac12\left[\|f\|^2 + \aa\|\de_y f\|^2+2\bb \l u'\de_x f, \de_y f \r+\cc \| u'\de_x f\|^2 \right].
\end{align}
Notice that, up to rescaling of the various coefficients, $\Phi$ has exactly the form \eqref{eq:PHIkk}, as long as we assume that 
$f$ is concentrated in one single frequency $k$.
Since
\begin{align}
2\bb |\l u'\de_x f, \de_y f \r|\leq \frac{\aa}{2}\|\de_y f\|^2+ \frac{2\bb^2}{\aa} \| u'\de_x f\|^2,
\end{align}
if we assume that
\begin{align}\label{eq:constraint}
\frac{\bb^2}{\aa\cc}\leq \frac{1}{4},
\end{align}
we have that
\begin{align}\label{eq:Phipos}
\frac12\left[\|f\|^2 + \frac{\aa}{2}\|\de_y f\|^2+\frac{\cc}{2} \| u'\de_x f\|^2 \right]\leq \Phi\leq\frac12\left[\|f\|^2 + \frac{3\aa}{2}\|\de_y f\|^2+\frac{3\cc}{2} \| u'\de_x f\|^2 \right] .
\end{align}
Collecting \eqref{eq:L2},  \eqref{eq:alpha},  \eqref{eq:beta}, and  \eqref{eq:gamma}, we find that
\begin{align}
&\ddt \Phi+\nu\| \de_y f\|^2+\aa \nu\|L f\|^2 +\bb \| u'\de_x f\|^2+\cc\nu \| u'\de_{xy} f\|^2\notag\\
&\qquad=-\aa\l u'\de_x f,\de_y f \r
-\bb\nu\left[2\l u'\de_{xy} f,Lf \r+\l u''\de_x f, Lf \r+\l v u'\de_x f, Lf \r\right]+2\cc\nu \l  u'' \de_{x} f, u'\de_{xy}f \r.
\end{align}
We now estimates the error terms one by one, possibly adding further constraints on $\aa,\bb,\cc$, and then verify that a suitable choice is possible. In what follows, $C_0\geq 2$ is a constant that depends on $u,v$ and that can change from line to line, but it is crucially independent of $\aa,\bb,\cc,\nu$ and the $x$-Fourier mode $k$. 
For the term containing $\aa$, we have
\begin{align}
\aa|\l u'\de_x f,\de_y f \r|\leq \frac{\bb}{2}\|u'\de_x f\|^2+ C_0 \frac{\aa^2}{\bb}\|\de_y f\|^2.
\end{align}
while the first $\bb$-error terms can be estimated as
\begin{align}
2\bb\nu|\l u'\de_{xy} f,Lf \r| \leq \frac{\cc\nu}{4}\|u'\de_{xy} f\|^2 +C_0 \frac{\bb^2\nu}{\cc}\|L f\|^2.
\end{align}
Using the boundedness of $u''$, we also have
\begin{align}
\bb\nu|\l u''\de_x f, Lf \r|\leq \cc\nu \|\de_x f\|^2 +C_0 \frac{\bb^2\nu}{\cc}\|L f\|^2,
\end{align}
while using the boundedness of $v$ we can derive the bound
\begin{align}
\bb\nu|\l v u'\de_x f, Lf \r|\leq \cc\nu \|\de_x f\|^2 +C_0 \frac{\bb^2\nu}{\cc}\|L f\|^2.
\end{align}
Finally, the $\cc$-error term is estimated as
\begin{align}
2\cc\nu| \l  u'' \de_{x} f, u'\de_{xy}f \r|\leq \frac{\cc\nu}{4}\|u'\de_{xy}f\|^2+ C_0\cc\nu \|\de_x f\|^2.
\end{align}
Collecting all the bounds above and assuming the  more stringent constraints
\begin{align}\label{eq:constraint2}
\frac{\bb^2}{\aa\cc}\leq \frac{1}{6C_0},
\end{align}
and
\begin{align}\label{eq:constraint3}
 \frac{\aa^2}{\bb}\leq \frac{\nu}{2C_0},
\end{align}
we arrive at
\begin{align}\label{eq:Diff1}
&\ddt \Phi+\frac{\nu}{2}\| \de_y f\|^2+\frac{\aa \nu}{2}\|L f\|^2 +\frac{\bb}{2} \| u'\de_x f\|^2+\frac{\cc\nu}{2} \| u'\de_{xy} f\|^2\leq
C_0\cc\nu \|\de_x f\|^2.
\end{align}
To estimate the right-hand side above, it is convenient to think of $\de_x=i k$ and rescale the parameters in the following way:
\begin{align}\label{eq:coeffchoice}
\aa=\frac{\nu^{1/2}}{k^{1/2}}\aa_0, \qquad \bb=\frac{1}{k}\bb_0, \qquad \cc=\frac{1}{\nu^{1/2}k^{3/2}}\cc_0,
\end{align}
with $\aa_0,\bb_0,\cc_0$ independent of $\nu,k$ and such that
\begin{align}\label{eq:constraint4}
\frac{\bb_0^2}{\aa_0\cc_0}\leq \frac{1}{6C_0} 
\end{align}
and
\begin{align}\label{eq:constraint5}
 \frac{\aa_0^2}{\bb_0}\leq \frac{1}{2C_0},
\end{align}
so that  the constraints \eqref{eq:constraint}, \eqref{eq:constraint2} and \eqref{eq:constraint3} are automatically satisfied, provided we
can choose $\aa_0,\bb_0,\cc_0$ as above. Re-writing \eqref{eq:Diff1}, we end up with
\begin{align}
&\ddt \Phi+\frac{\nu}{2}\| \de_y f\|^2 +\frac{\bb_0}{2} \frac{1}{k} \| u'\de_x f\|^2+\frac{\aa}{2} \|L f\|^2+\frac{\cc}{2}\| u'\de_{xy} f\|^2\leq
C_0\cc_0\nu^{1/2}k^{1/2} \|f\|^2.
\end{align}
We now make use of the following inequality, derived in \cite{BCZ15}*{Proposition 2.7}, 
\begin{align}
\sigma\|g\|^2\leq C_0\left[\sigma^2 \|\de_yg\|^2+ \| u'g\|^2\right],
\end{align}
valid for any function $g:\T\to \C$ in $H^1$ and for any sufficiently small $\sigma>0$. In our case, we will apply it with
$g=f$, identifying $f$ with its $k$-th Fourier mode, and 
\begin{align}
\sigma^2=\frac{1}{\bb_0}\frac{\nu}{k}\leq\frac{\nu_0}{\bb_0} \ll 1,
\end{align}
provided we choose $\nu_0\ll \bb_0$.
In this way
\begin{align}\label{eq:spect}
\nu^{1/2}k^{1/2}  \|f\|^2\leq \frac{C_0}{\bb_0^{1/2}}\left[\nu\|\de_yf\|^2+ k\| u'f\|^2\right]=  \frac{C_0}{\bb_0^{1/2}}\left[\nu\|\de_yf\|^2+ \frac{\bb_0}{k}\| u'\de_x f\|^2\right].
\end{align}
Thus, from \eqref{eq:Diff3}, the above inequality and the further constraint 
\begin{align}\label{eq:constraint6}
\frac{\cc_0}{\bb_0^{1/2}}\leq \frac{1}{2C_0},
\end{align}
we learn that
\begin{align}\label{eq:Diff3}
&\ddt \Phi+\frac{\nu}{4}\| \de_y f\|^2+\frac{\bb_0}{4} \frac{1}{k} \| u'\de_x f\|^2+\frac{\aa}{2} \|L f\|^2+\frac{\cc}{2}\| u'\de_{xy} f\|^2 \leq 0.
\end{align}
We use \eqref{eq:spect} once more to deduce that
\begin{align}
&\ddt \Phi+\frac{\bb_0^{1/2}}{C_0} \nu^{1/2}k^{1/2} \|f\|^2+ \frac{\nu}{4}\| \de_y f\|^2 +\frac{\bb_0}{4} \frac{1}{k} \| u'\de_x f\|^2+\frac{\aa}{2} \|L f\|^2+\frac{\cc}{2}\| u'\de_{xy} f\|^2\leq 0.
\end{align}
Equivalently, we can make the decay rate explicit by writing
\begin{equation}\label{eq:Diff4}
\ddt \Phi+\nu^{1/2}k^{1/2}\left[\frac{\bb_0^{1/2}}{C_0}  \|f\|^2+ \frac{\aa}{4\aa_0}\| \de_y f\|^2 +\frac{\bb_0}{4\cc_0} \cc \| u'\de_x f\|^2\right]+\frac{\aa}{2} \|L f\|^2+\frac{\cc}{2}\| u'\de_{xy} f\|^2 \leq 0.
\end{equation}
We now choose $\aa_0,\bb_0,\cc_0$ complying with \eqref{eq:constraint4}, \eqref{eq:constraint5} and \eqref{eq:constraint6}.
Let 
\begin{align}
\delta_0=\left[\frac{1}{288 C_0^3}\right]^{1/4},
\end{align}
and set
\begin{align}
\aa_0=12 C_0\delta_0^3, \qquad \bb_0=\delta_0^2, \qquad  \cc_0=\frac{\delta_0}{2 C_0}.
\end{align}
In this way, 
\begin{align}
\frac{\bb_0^2}{\aa_0\cc_0}= \frac{1}{6C_0}, \qquad \frac{\cc_0}{\bb_0^{1/2}}= \frac{1}{2C_0}, \qquad \frac{\aa_0^2}{\bb_0}=\frac{1}{2C_0},
\end{align}
so that  \eqref{eq:constraint4}, \eqref{eq:constraint5} and \eqref{eq:constraint6} are automatically satisfied, and since
\begin{align}
\frac{\bb_0^{1/2}}{C_0}= \frac{\delta_0}{C_0}, \qquad \frac{1}{4\aa_0}=\frac{1}{12C_0\delta_0^3},\qquad \frac{\bb_0}{4\cc_0}=\frac{C_0}{2}\delta_0,
\end{align}
we deduce from \eqref{eq:Diff4} that
\begin{align}
&\ddt \Phi+\frac{\delta_0}{C_0}\nu^{1/2}k^{1/2}\left[  \|f\|^2+ \frac{\aa}{12\delta_0^4}\| \de_y f\|^2 +\frac{C_0^2}{2} \cc \| u'\de_x f\|^2\right]+\frac{\aa}{2} \|L f\|^2+\frac{\cc}{2}\| u'\de_{xy} f\|^2 \leq 0.
\end{align}
Consequently, since $C^2_0\geq 3$ and $\delta_0^4\leq 1/18$, we arrive at
\begin{align}
&\ddt \Phi+\frac{\delta_0}{C_0}\nu^{1/2}k^{1/2}\left[  \|f\|^2+ \frac{3\aa}{2}\| \de_y f\|^2 +\frac{3\cc}{2}  \| u'\de_x f\|^2\right]+\frac{\aa}{2} \|L f\|^2+\frac{\cc}{2}\| u'\de_{xy} f\|^2 \leq 0.
\end{align}
Defining
\begin{align}
\eps_0=\frac{2\delta_0}{C_0},
\end{align}
and using \eqref{eq:Phipos}, we finally obtain the differential inequality
\begin{align}
\ddt \Phi+\eps_0\nu^{1/2}k^{1/2}\Phi+\frac{\aa}{2} \|L f\|^2+\frac{\cc}{2}\| u'\de_{xy} f\|^2 \leq 0.
\end{align}
This concludes the proof of Theorem \ref{thm:main3}.
%
%
\subsection{Estimates in $L^2_{\rho_\infty}$}\label{sub:semigroup}
We prove here Corollary \ref{cor:main3}. 
We first prove that 
\begin{align}\label{eq:L2kLARGE}
\|f(t)\|^2 \leq  C_0\| f^{in}\|^2 \e^{-\eps_0\frac{\nu^{1/2}k^{1/2}}{1+|\ln\nu|+\ln k}t} , \qquad \forall t\geq T_{\nu,k}:=\frac{1+|\ln\nu|+\ln k }{\eps_0\nu^{1/2}k^{1/2}},
\end{align}
where again we neglect the dependence on $k$ of $f$ and $C_0\geq 1$ is some constant.
From \eqref{eq:L2} and the mean value theorem, it is easy to see that there exists 
\begin{align}
t_\star\in\left(0,\frac{1}{\eps_0\nu^{1/2}k^{1/2}}\right)
\end{align}
such that
\begin{align}\label{eq:parabreg}
\frac{\nu^{1/2}}{k^{1/2}}\|\de_yf (t_\star)\|^2\leq \frac{\eps_0}{2} \|f^{in}\|^2.
\end{align}
Moreover, from \eqref{eq:diffPHI2}, we deduce that
\begin{align}\label{eq:fin1}
\Phi(t) \leq  \e^{-\eps_0\nu^{1/2}k^{1/2}t}\Phi(t_\star), \qquad \forall t\geq t_\star.
\end{align}
Now, using \eqref{eq:Phipos} and  \eqref{eq:coeffchoice}, we see that
\begin{align}\label{eq:Phiposk}
\frac12\|f\|^2 \leq \Phi\leq C_0\left[\|f\|^2 +  \frac{\nu^{1/2}}{k^{1/2}}\|\de_y f\|^2+\frac{k^{1/2}}{\nu^{1/2}} \|f\|^2 \right] .
\end{align}
Therefore,  using \eqref{eq:parabreg} and \eqref{eq:L2} we arrive at
\begin{align}
\Phi(t_\star)\leq C_0\left[\|f(t_\star)\|^2+\|f^{in}\|^2 +\frac{k^{1/2}}{\nu^{1/2}} \|f(t_\star)\|^2 \right]\leq C_0\left[\|f^{in}\|^2 +\frac{k^{1/2}}{\nu^{1/2}} \|f^{in}\|^2 \right].
\end{align}
By noticing that since $\nu_0\ll 1$ there holds
\begin{align}
\frac{k^{1/2}}{\nu^{1/2}}\e^{-\eps_0\nu^{1/2}k^{1/2}t}\leq \e^{-\eps_0\frac{\nu^{1/2}k^{1/2}}{1+|\ln\nu|+\ln k}t},\qquad \forall t\geq T_{\nu,k},
\end{align}
we then find from \eqref{eq:fin1} that for all $t\geq t_\star$ there holds
\begin{align}
\Phi(t)\leq \Phi(t_\star) \e^{-\eps_0\nu^{1/2}k^{1/2}(t-t_\star)}
&\leq C_0\left[\|f^{in}\|^2 +\frac{k^{1/2}}{\nu^{1/2}} \|f^{in}\|^2 \right] \e^{-\eps_0\nu^{1/2}k^{1/2}(t-t_\star)}\notag\\
&= C_0\e^{\eps_0\nu^{1/2}k^{1/2}t_\star} \left[\|f^{in}\|^2 +\frac{k^{1/2}}{\nu^{1/2}} \|f^{in}\|^2 \right] \e^{-\eps_0\nu^{1/2}k^{1/2}t}\notag\\
&\leq C_0 \left[\|f^{in}\|^2 +\frac{k^{1/2}}{\nu^{1/2}} \|f^{in}\|^2 \right] \e^{-\eps_0\nu^{1/2}k^{1/2}t}\notag\\
&\leq C_0\| f^{in}\|^2 \e^{-\eps_0\frac{\nu^{1/2}k^{1/2}}{1+|\ln\nu|+\ln k}t},
\end{align}
which in particular implies \eqref{eq:L2kLARGE} upon using \eqref{eq:Phiposk} once more. Now, to extend \eqref{eq:Phiposk} to all $t\in[0,T_{\nu,k})$,
we simply notice that  $\| f(t)\|$ is decreasing and
\begin{align}
\min_{t\in[0,T_{\nu,k}]}C_0 \e^{-\eps_0\frac{\nu^{1/2}k^{1/2}}{1+|\ln\nu|+\ln k}t}=C_0 \e^{-\eps_0\frac{\nu^{1/2}k^{1/2}}{1+|\ln\nu|+\ln k}T_{\nu,k}}= C_0\e^{-1}\leq C_0,
\end{align}
concluding the proof of Corollary  \ref{cor:main3}.
%
%
\section{Numerical simulations}\label{sec:numerics}

In this section we illustrate some of the theoretical results obtained in the previous sections by means of some simple numerical experiments. In particular, we confirm the scalings for the diffusion coefficient, as a function of the strength of the gradient perturbation, for a simple two-dimensional shear flow and we also explore this scaling in the presence of closed streamlines, i.e. cat's eye flows and cellular flows.

We consider the two-dimensional Langevin dynamics, consistent with the rescaled SDEs~\eqref{eq:SDEsystem0},
\begin{align}\label{e:lang_numer}
\dd \XX^{\eps,\nu}(t) = \left[ \frac{1}{\nu} \nabla^{\perp} \psi_\eps(\XX^{\eps,\nu}(t)) - \nabla \psi_\eps(\XX^{\eps,\nu}(t)) \right]\dd t + \sqrt{2\kappa} \, \dd \WW(t),
\end{align}
where $\nabla^{\perp} = \left(-\de_y,\de_x \right)$ and for  the one-parameter family of stream functions  
\begin{equation}\label{e:stream_function}
\psi_{\eps}({\bx})= \eps \sin (3x)+ \sin (3y) , \quad \mbox{with} \;\; \eps \in[0,1].
\end{equation}
This family of stream functions can be mapped to the corresponding Childress-Soward flow~\cite{kramer}, given by 
$\tilde{\psi}_{\eps}(x,y) = \sin(3x) \sin(3y) + \eps \cos(3x) \cos(3y)$, via an appropriate rotation/change of coordinates. 
Typical streamlines for different values of $\eps$ are shown in Figure~\ref{fig:stream}.
\begin{figure}[h!]
  \centering
  \begin{subfigure}[b]{0.32\linewidth}
    \includegraphics[width=\linewidth]{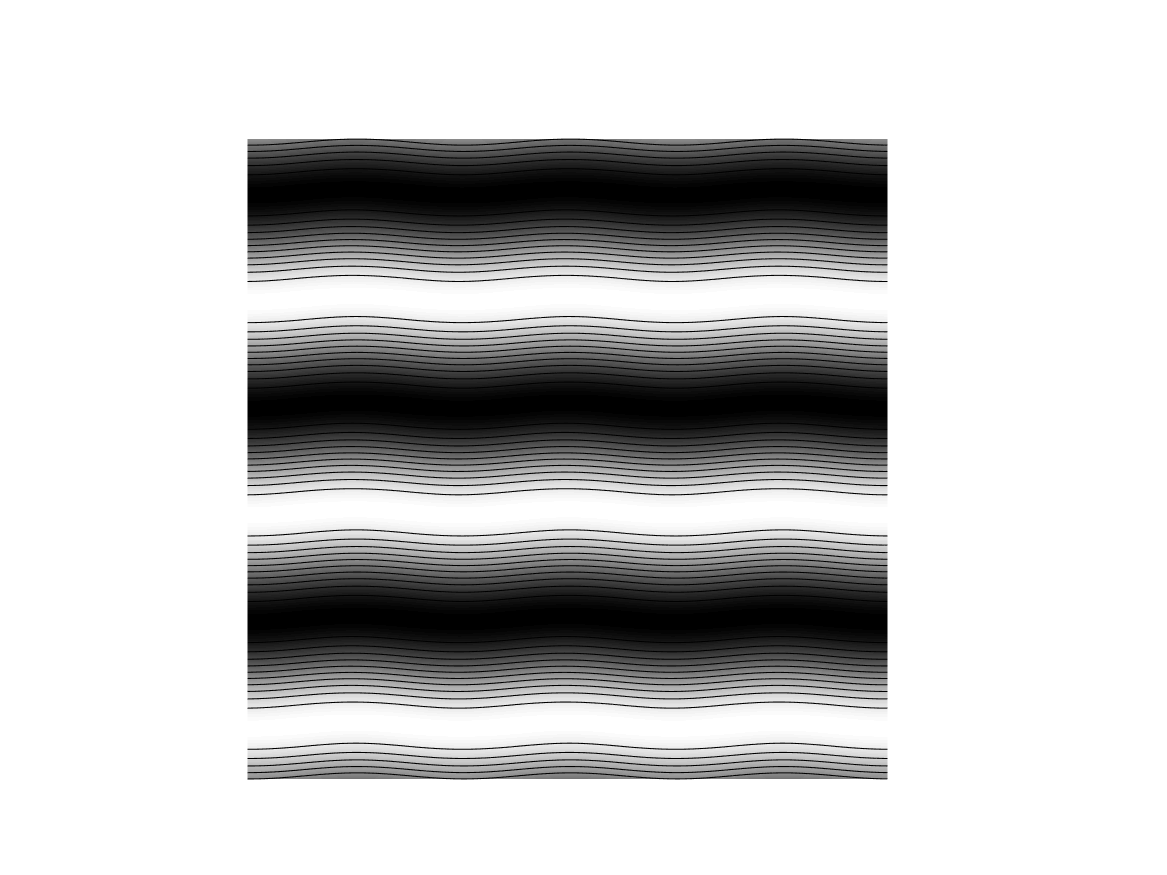}
     \caption{$\eps=0$}
  \end{subfigure}
  \begin{subfigure}[b]{0.32\linewidth}
    \includegraphics[width=\linewidth]{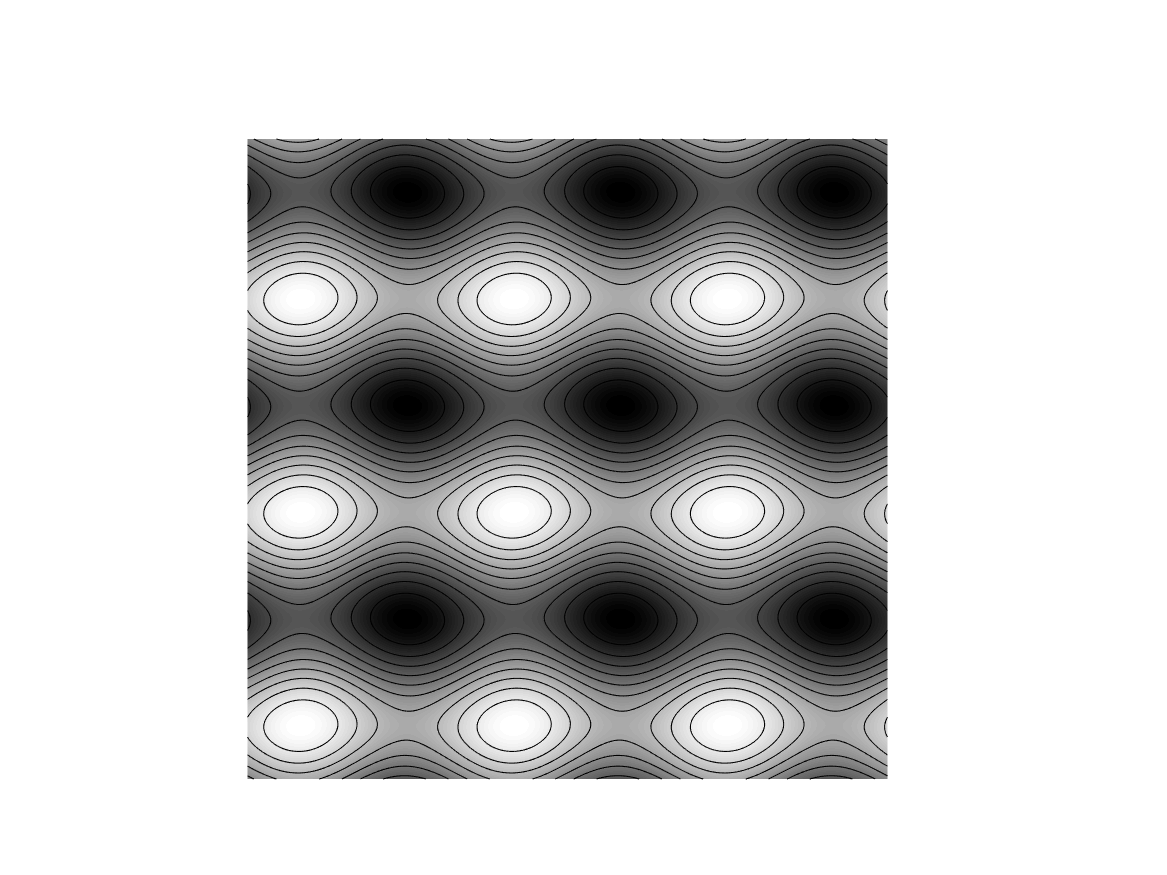}
     \caption{$\eps=0.5$}
  \end{subfigure}\label{fig:stream05}
  \begin{subfigure}[b]{0.32\linewidth}
    \includegraphics[width=\linewidth]{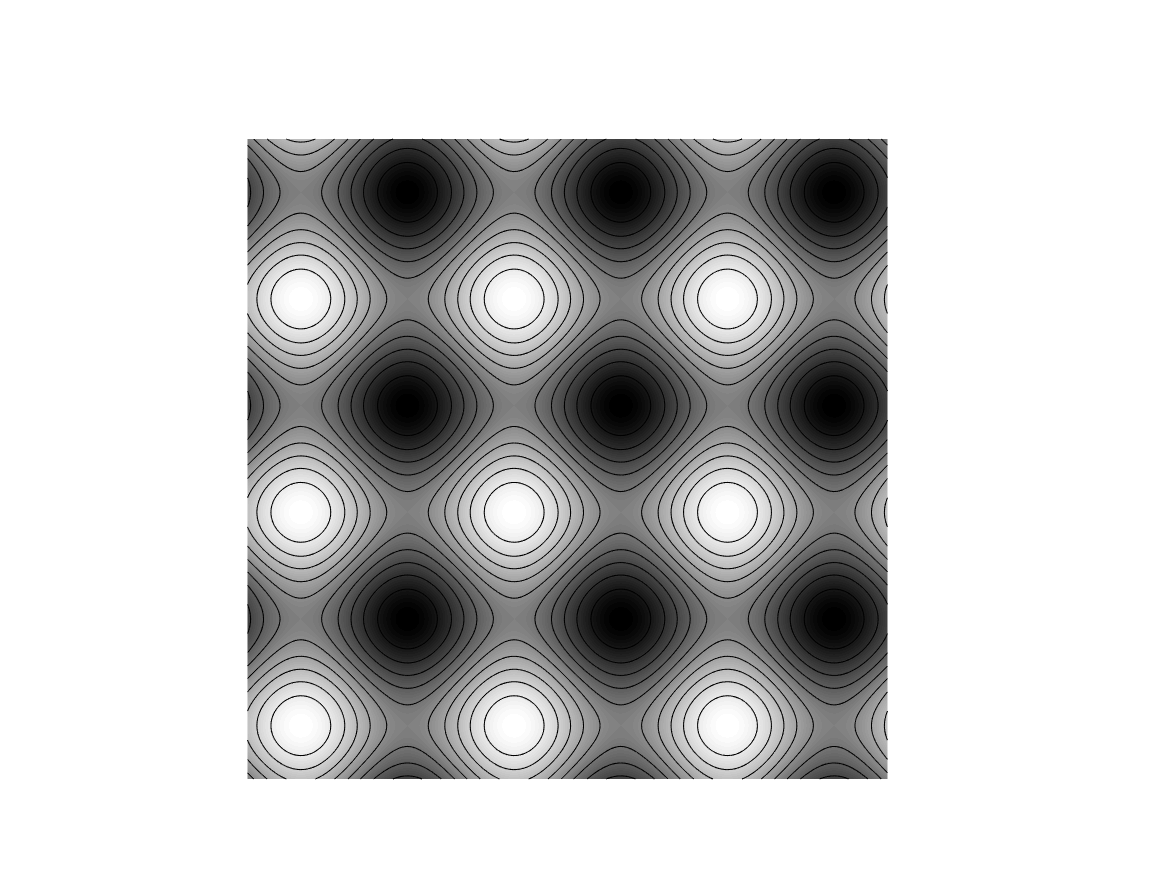}
     \caption{$\eps=1$}
  \end{subfigure}
\caption{Streamlines of the stream function $\psi(x,y) = \eps\sin(3 x) + \sin(3 y)$, for $\eps=0,0.5,1$.}
  \label{fig:stream}
\end{figure}

\subsection{Effective diffusivity}
We calculate the effective diffusion tensor using Monte Carlo simulations \cites{McL98, PavlStBan06} and using the Langrangian definition of the effective diffusivity:
\begin{equation}
D_{eff} = \lim_{t \rightarrow +\infty} \frac{(\XX^{\eps,\nu}(t)-\E \XX^{\eps,\nu}(t)) \otimes (\XX^{\eps,\nu}(t)-\E \XX^{\eps,\nu}(t)) }{2 t}.
\end{equation}
Equations~\eqref{e:lang_numer}-\eqref{e:stream_function}, read, with $\XX^{\eps,\nu} = (X(t), \, Y(t))$ and $\WW(t) = (W_1(t), \, W_2(t))$:
\begin{align}\label{eq:XXXXX}
\dd X(t) &= -\frac{3}{\nu} \cos(3 Y(t)) \, \dd t - 3 \eps \cos(3 X(t)) \, \dd t + \sqrt{2 \kappa} \, \dd W_1(t),  
\\ 
\dd Y(t) &= \frac{3}{\nu} \eps \cos(3 X(t)) \, \dd t - 3  \cos(3 Y(t)) \, \dd t + \sqrt{2 \kappa} \, \dd W_2(t).
\end{align}
We solve the SDEs~\eqref{e:lang_numer} using the Euler-Maruyama scheme. An alternative is 
the numerical method developed in~\cite{PavlStZyg09} that is particularly tailored to the calculation of
the eddy diffusivity for periodic vector fields, and that is computationally more efficient in
the small $\kappa$ regime. Given that our primary focus is on the study of the effect of the
compressible perturbation on the eddy diffusivity, in our numerical experiments we set $\kappa =1$. We also consider the regime $\nu \in (0,1)$, since for values of $\nu$ of $O(1)$ and larger, dynamic is dominated by molecular diffusion.
 
For $\eps=0$, the effective diffusion coefficient in the $y$ direction is independent of $\nu$ and is given by the standard Lifson-Jackson formula for the diffusion coefficient of a Brownian particle moving in a one dimensional periodic potential~\cite{pavliotis2014book}*{Eqn. (13.6.13)}. In Figure~\ref{fig:bm-pathsA} we plot the effective diffusion coefficient along the $x$ direction. Using Theorem \ref{thm:main1} and an appropriate rescaling, it is easy to check that the effective diffusion coefficient in the $x$ direction, and in the absence of noise in the $x$ direction, is $D_{xx}= D_0\nu^{-2}$, where $D_0$ is given by the Lifson-Jackson formula for the potential $\psi_{0}=\sin(3 y)$.  Thus it is in agreement with the slope in this figure. 

At $\eps=0.5$, we have both open and closed streamlines, and the diffusion coefficient scales differently in the $x$ and $y$ directions, as a function of $\nu$. In fact, the diffusion coefficient in the $y$ direction, $D_{yy}$, depends weakly on $\eps$ and we do not present the plot. The diffusion coefficient $D_{xx}$  for $\eps=0.5$ is presented in Figure~\ref{fig:bm-pathsB}. Finally, we plot the diffusion coefficient for $\eps=1$, the case of closed streamlines. Due to symmetries, the diffusion coefficient is the same in the $x$ and $y$ directions, and we only plot $D_{xx}$ in Figure~\ref{fig:bm-pathsD}. Based on our numerical experiments for different values of $\eps$, not presented in this section, and elementary least squares fitting, we conjecture that the diffusion in the $x$ direction scales with $\nu$, for $\nu \in (0,1)$ as $D_{xx} \sim \nu^{-2 + \eps}$, at least for $\eps \in [0,1)$, but not necessarily at $\eps=1$. We will return to this conjecture in future work.

\begin{figure}[h!]
  \centering
  \begin{subfigure}[b]{0.32\linewidth}
    \includegraphics[width=\linewidth]{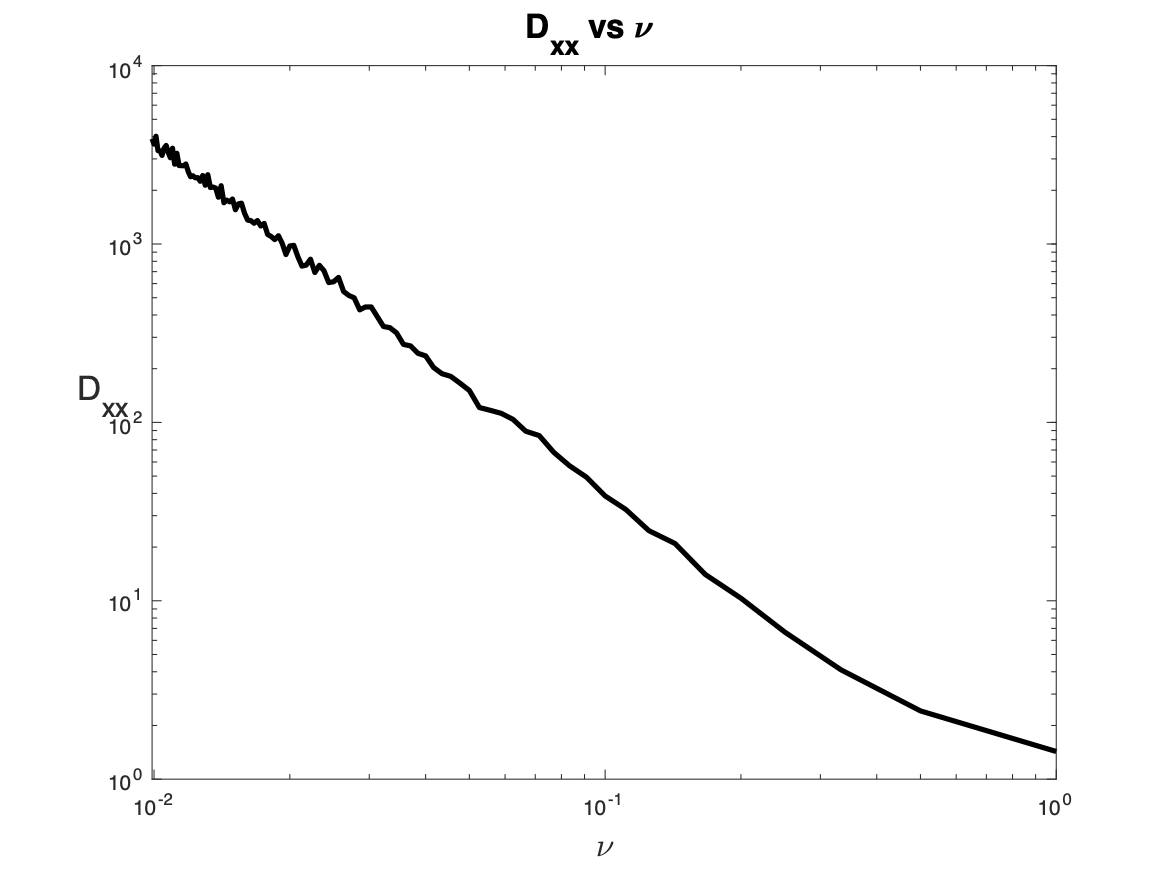}
    \caption{$\eps=0$}
    \label{fig:bm-pathsA}
  \end{subfigure}  
  \begin{subfigure}[b]{0.32\linewidth}
    \includegraphics[width=\linewidth]{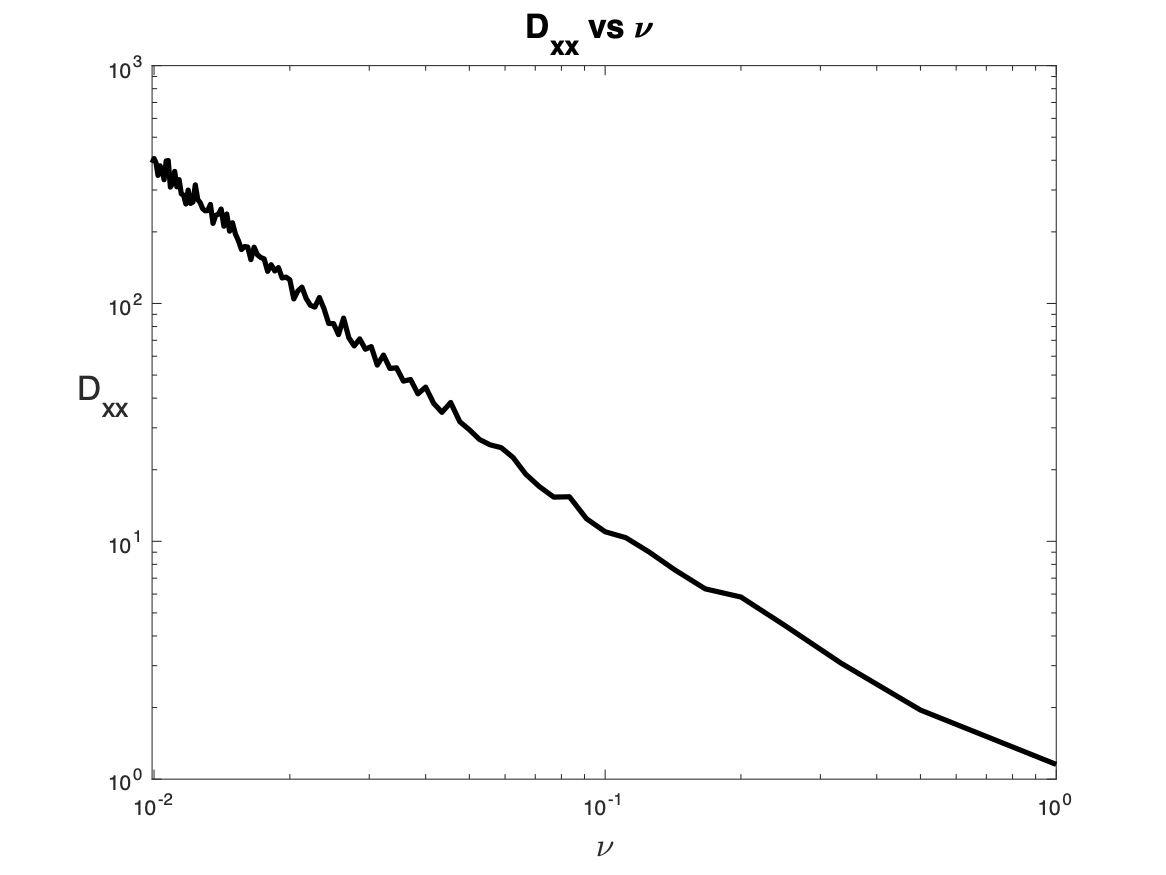}
    \caption{$\eps=0.5$}
        \label{fig:bm-pathsB}
  \end{subfigure}
  \begin{subfigure}[b]{0.32\linewidth}
    \includegraphics[width=\linewidth]{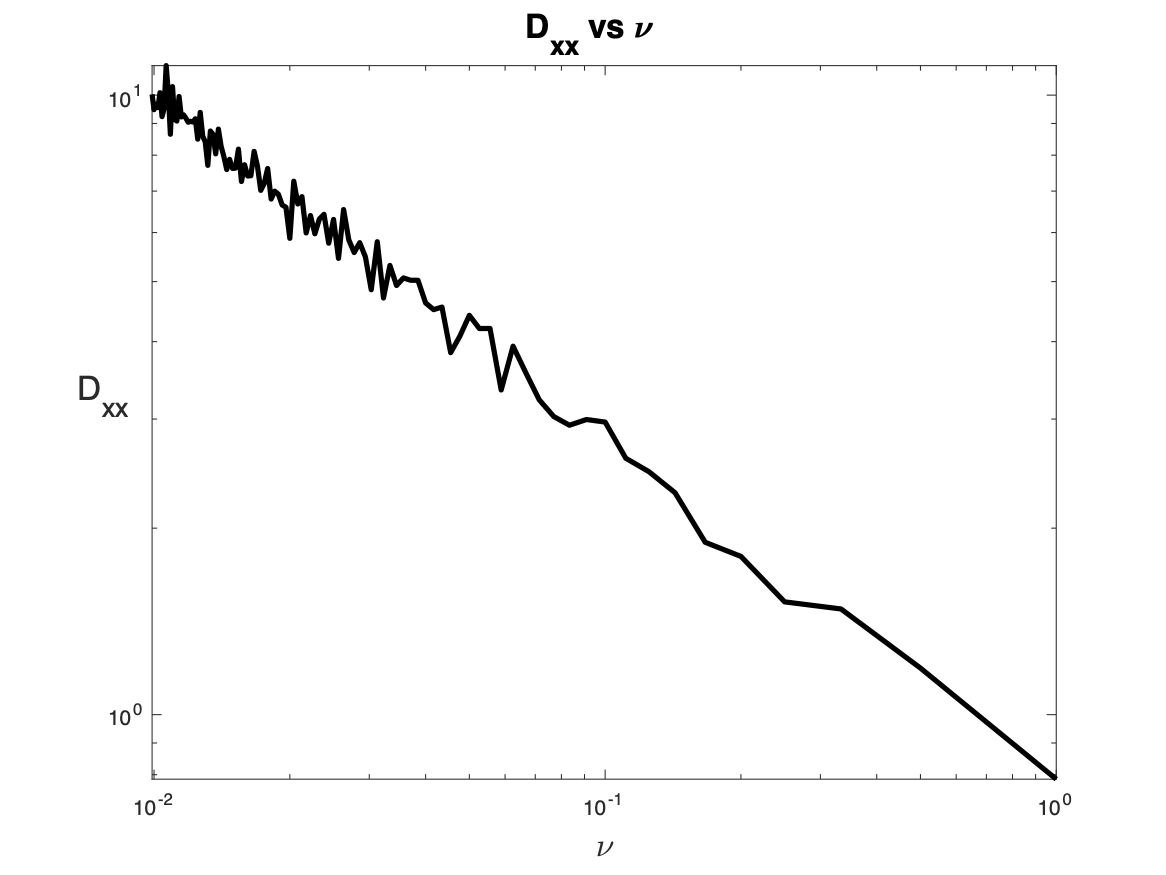}
    \caption{$\eps = 1$}
        \label{fig:bm-pathsD}
  \end{subfigure}
  \caption{The eddy diffusivity in the $x$ direction for $\eps=0, 0.5,1$.}
\end{figure}

\subsection{Enhanced diffusion}
Turning to the enhanced diffusion problem, when $\eps=1/2$, the stream function in \eqref{e:stream_function} has both open and closed level sets (see Figure \ref{fig:stream}). From the 
 mixing and relaxation enhancing perspective, the situation is depicted in Figure \ref{fig:cell05}. As in \eqref{FP:h}, the function $h$, solution to
 \begin{align}\label{FP:heps}
\de_t h+\frac{1}{\nu}  \nabla^{\perp} \psi_\eps \cdot\nabla h=\Delta h-\nabla  \psi_\eps\cdot\nabla h,\qquad h(0,\xx)=h^{in}(\xx),
\end{align}
 concentrates, on very short time-scales, on the streamlines of $\psi_\eps$.  Hence mixing \emph{along} streamlines is most efficient on a time scale
 shorter than the natural diffusive one (of order 1 in this case). For longer times instead diffusion takes over and dissipation happens mainly
\emph{across streamlines}. Although this is similar to what happen in the case $\eps=0$ (see Figure \ref{fig:shear}), the analysis here is much more complicated due to nontrivial  symmetries of the flow.

\begin{figure}[h!]
  \centering
  \begin{subfigure}[b]{0.24\linewidth}
    \includegraphics[width=\linewidth]{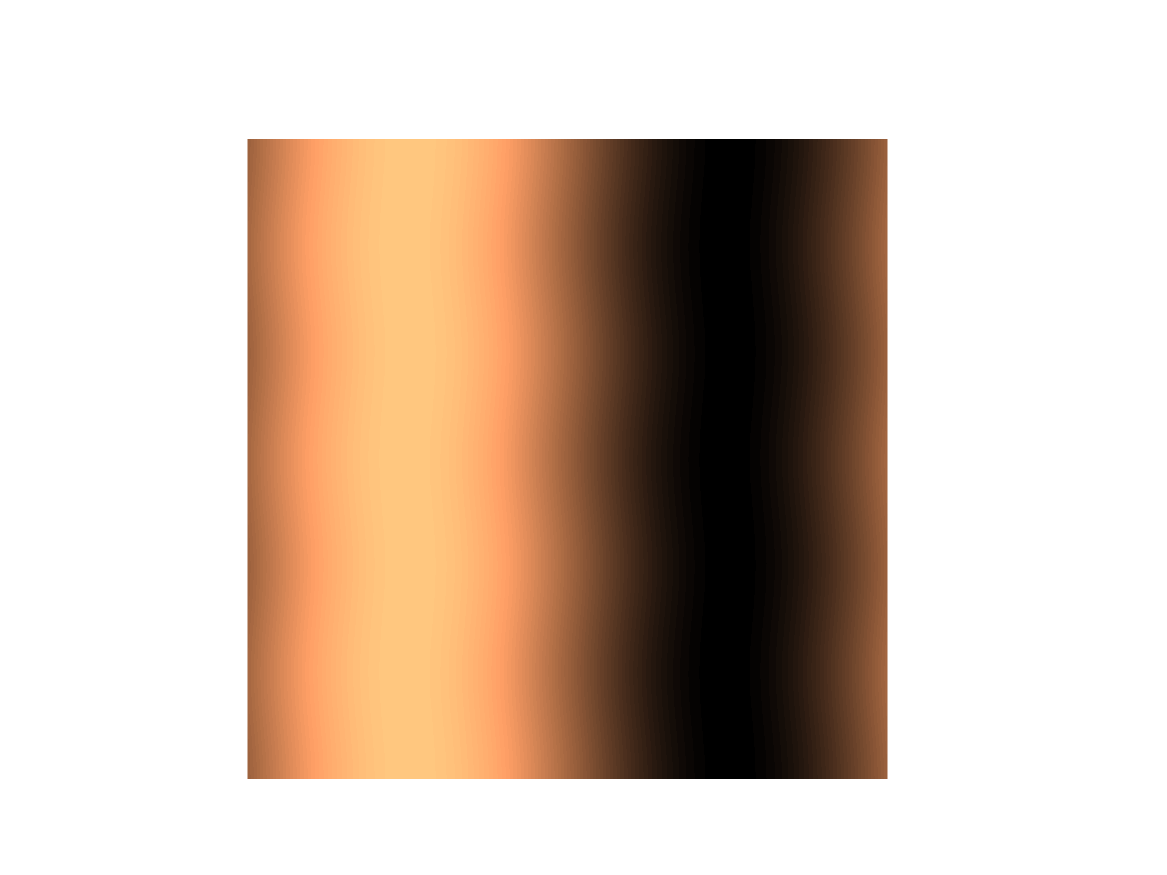}
  \end{subfigure}
  \begin{subfigure}[b]{0.24\linewidth}
    \includegraphics[width=\linewidth]{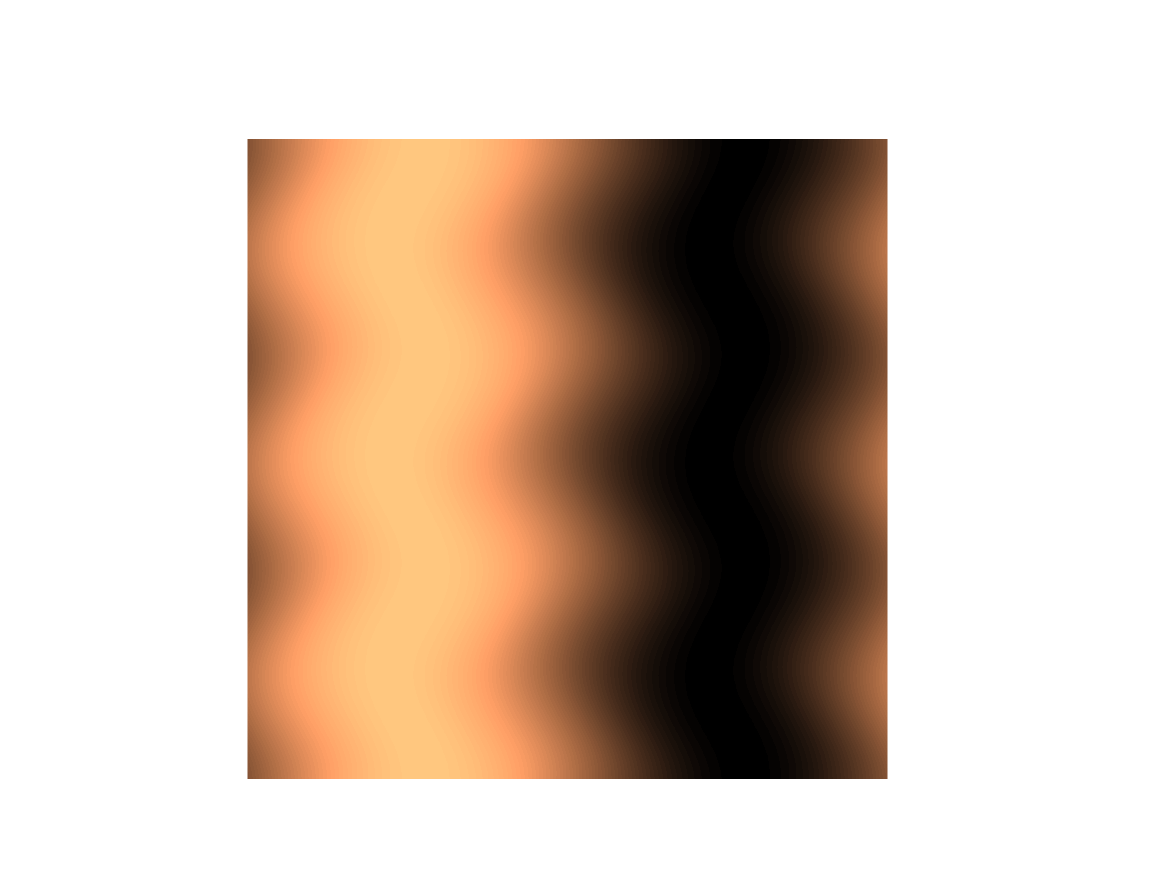}
  \end{subfigure}
  \begin{subfigure}[b]{0.24\linewidth}
    \includegraphics[width=\linewidth]{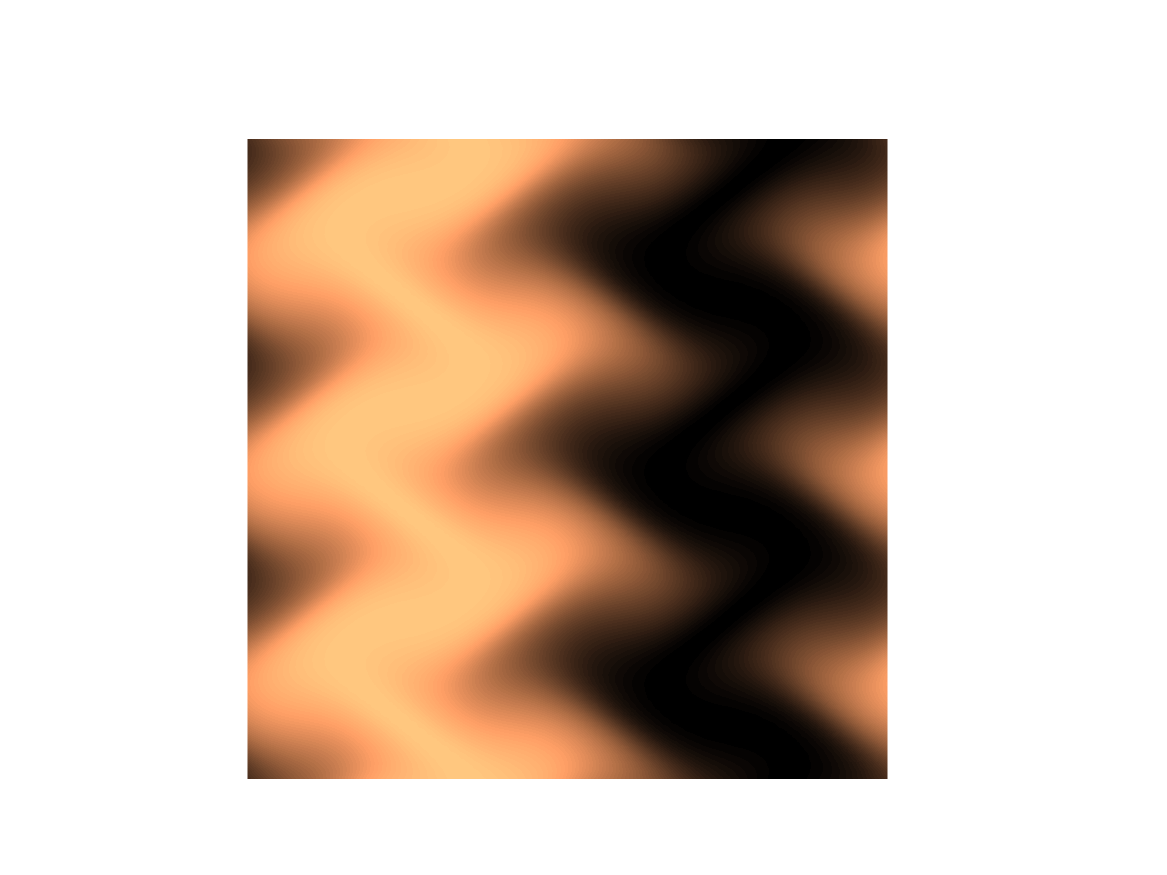}
  \end{subfigure}
  \begin{subfigure}[b]{0.24\linewidth}
    \includegraphics[width=\linewidth]{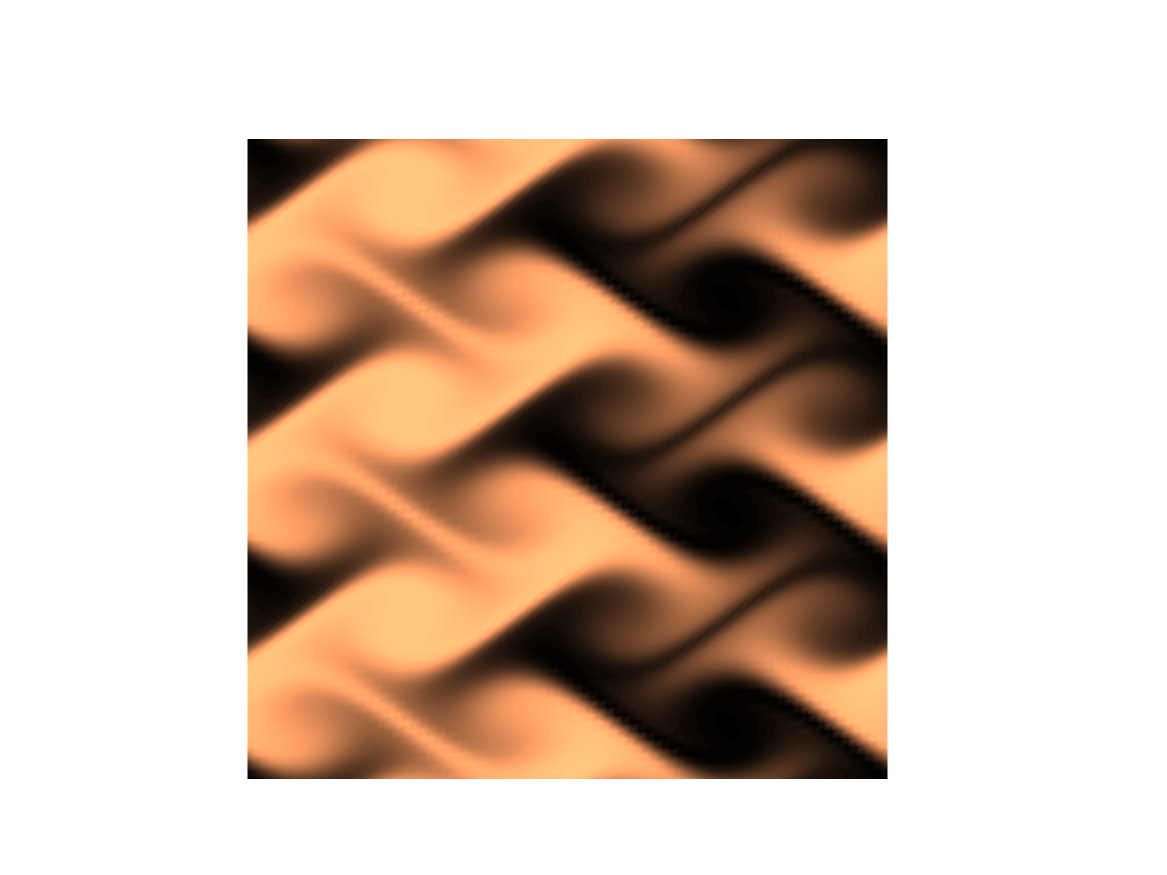}
  \end{subfigure}
    \begin{subfigure}[b]{0.24\linewidth}
    \includegraphics[width=\linewidth]{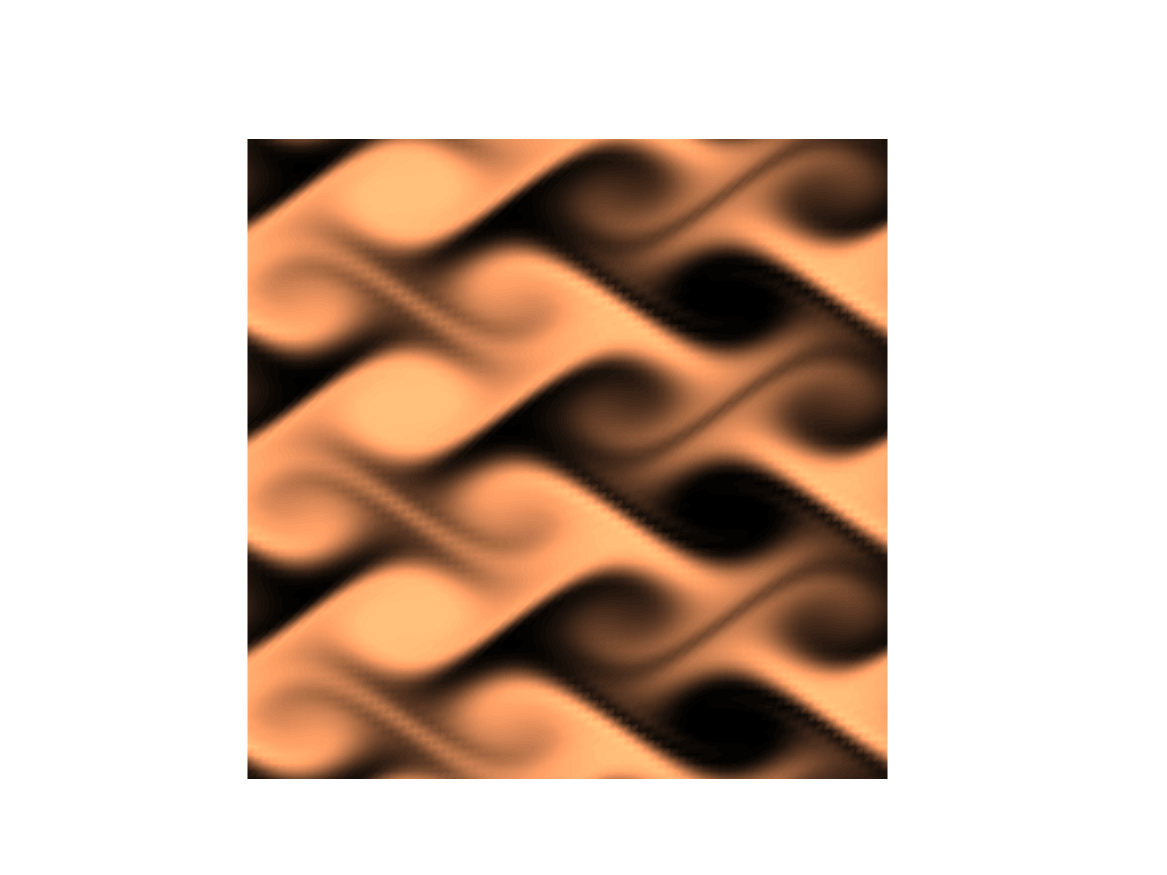}
  \end{subfigure}
  \begin{subfigure}[b]{0.24\linewidth}
    \includegraphics[width=\linewidth]{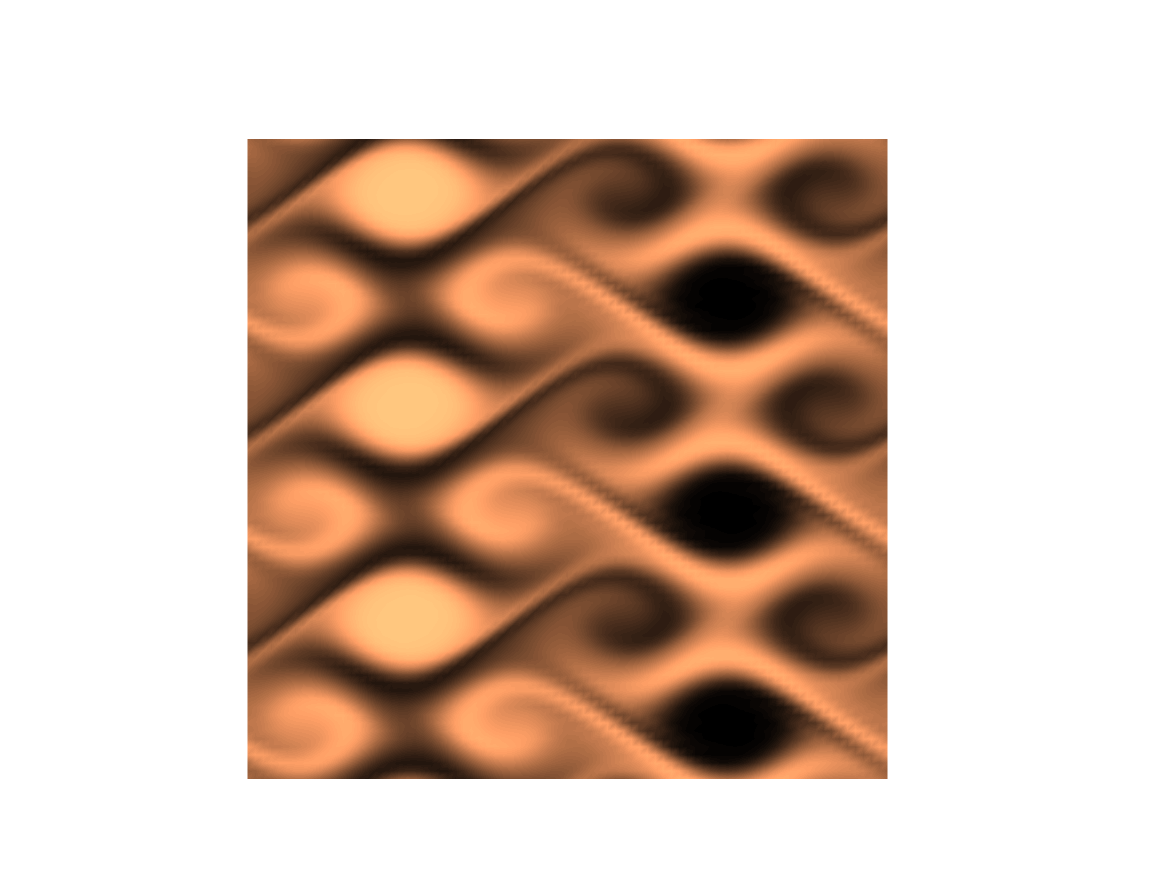}
  \end{subfigure}
  \begin{subfigure}[b]{0.24\linewidth}
    \includegraphics[width=\linewidth]{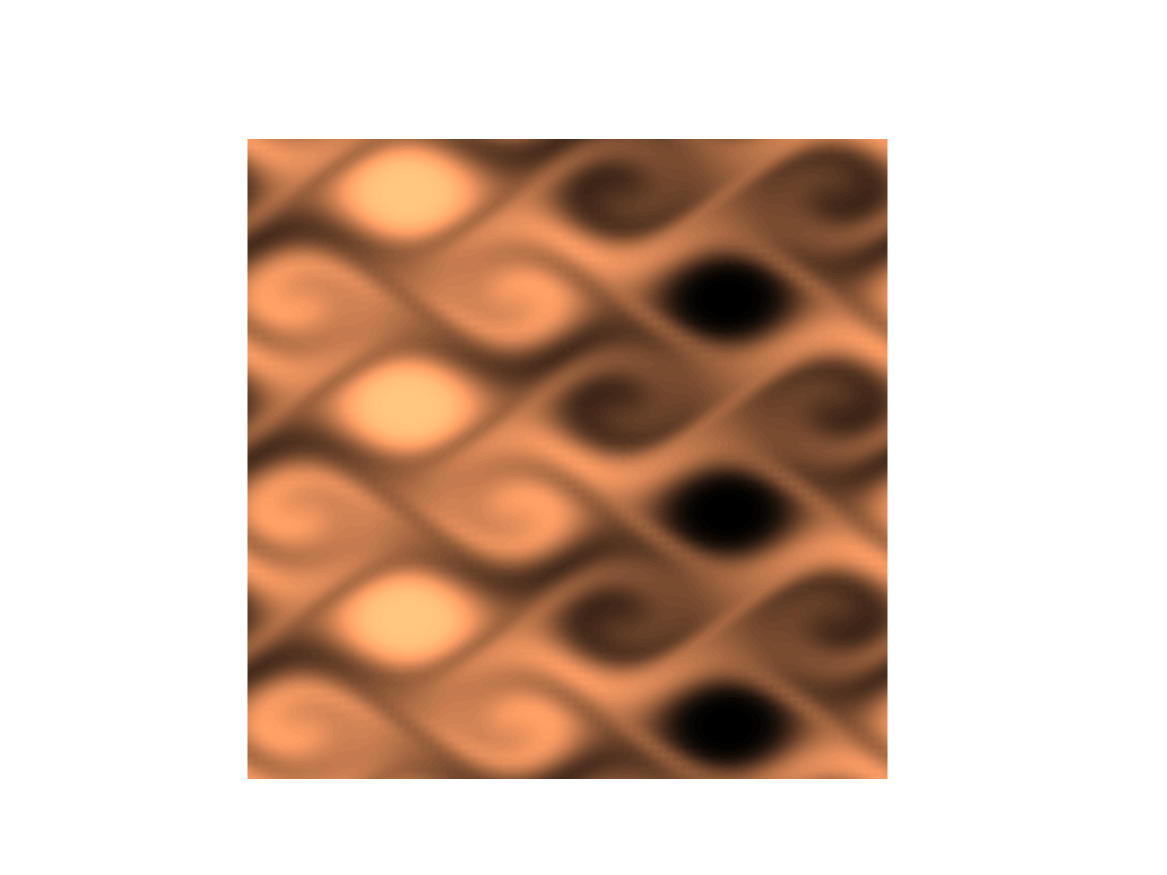}
  \end{subfigure}
  \begin{subfigure}[b]{0.24\linewidth}
    \includegraphics[width=\linewidth]{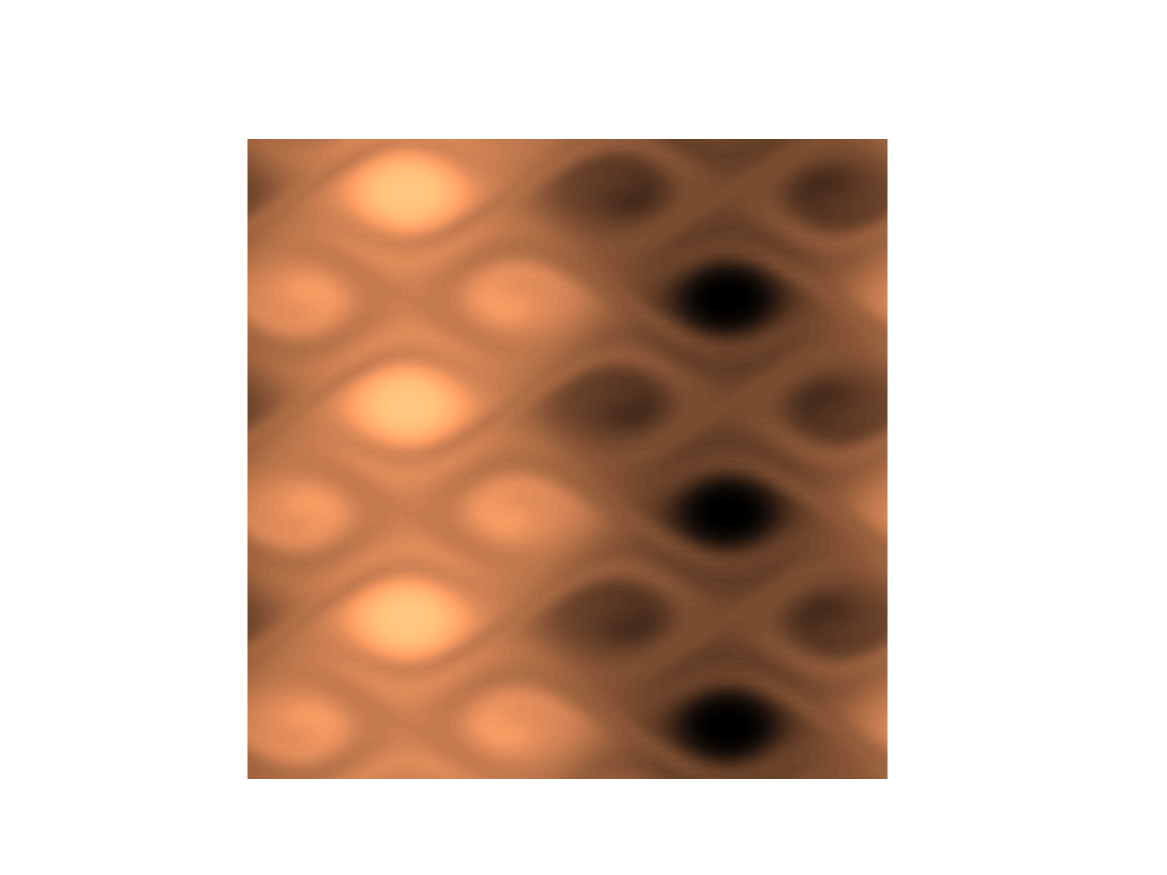}
  \end{subfigure}
  \caption{Snapshots of the evolutions of the solution $h$ to \eqref{FP:heps} when $\psi_\eps$ is given in \eqref{e:stream_function} with $\eps=1/2$.}
  \label{fig:cell05}
\end{figure}

A similar situation appears also for $\eps=1$ (see Figure \ref{fig:cell1}), where $h$ now 
 follows the streamlines in Figure \ref{fig:stream}, which are closed except for the hyperbolic manifolds. It is not clear how the dependence on 
 $\eps$ is reflected on an estimate of the type \eqref{eq:L2hreal}. Of course, the $x$-average should be replaced by an average on streamlines. However,
 it is not clear if the decay rate in $\nu$ will undergo significant changes. Although mixing is very fast near hyperbolic points, a global rate is
 very likely to be similar to that  in \eqref{eq:L2hreal}, since diffusion will ``push'' the solution $h$ into cells, away from the hyperbolic points.

\begin{figure}[h!]
  \centering
  \begin{subfigure}[b]{0.24\linewidth}
    \includegraphics[width=\linewidth]{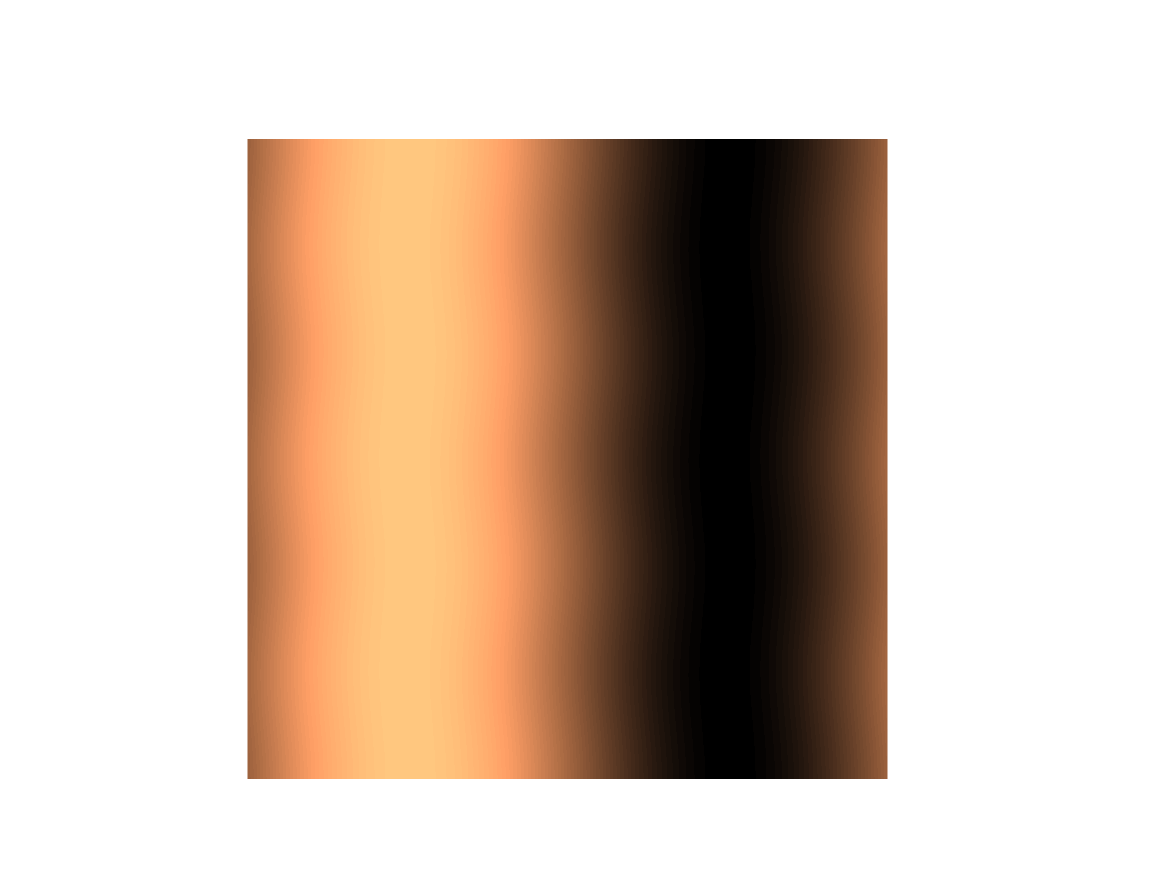}
  \end{subfigure}
  \begin{subfigure}[b]{0.24\linewidth}
    \includegraphics[width=\linewidth]{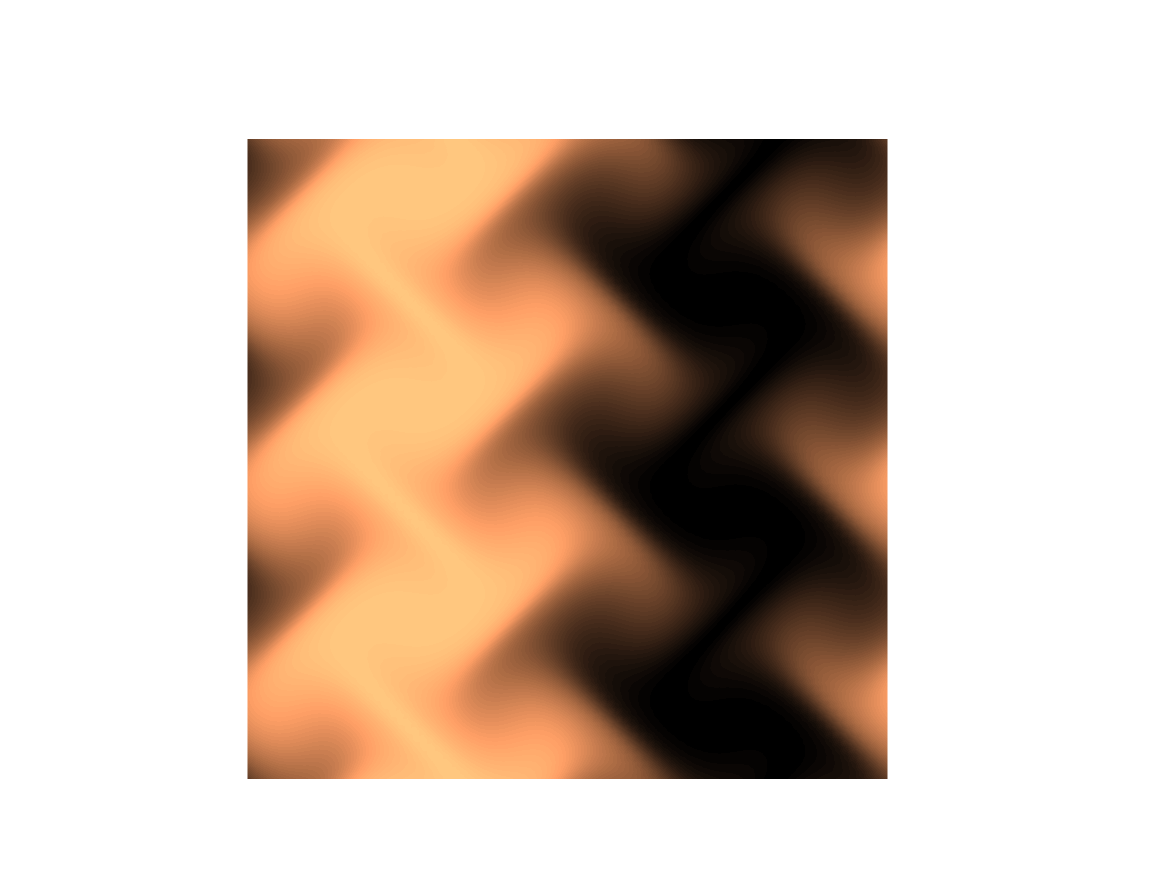}
  \end{subfigure}
  \begin{subfigure}[b]{0.24\linewidth}
    \includegraphics[width=\linewidth]{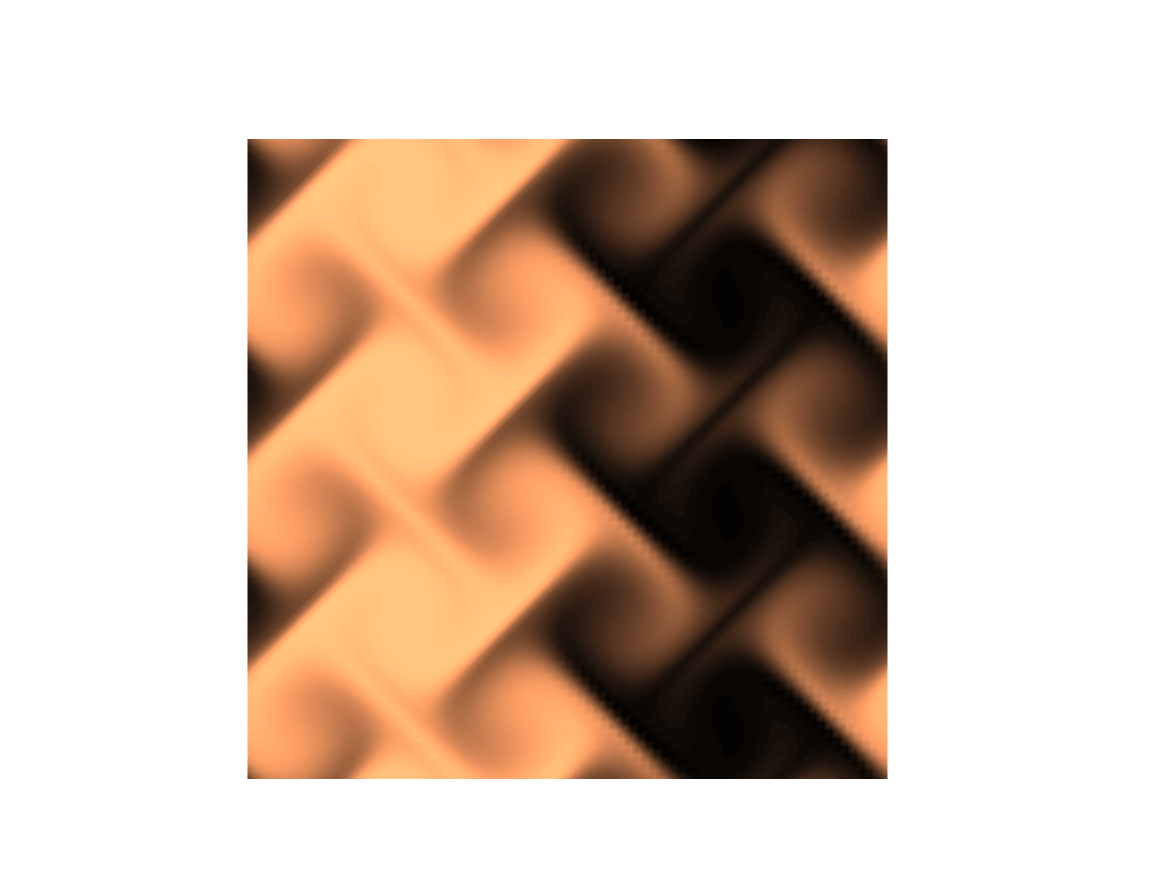}
  \end{subfigure}
  \begin{subfigure}[b]{0.24\linewidth}
    \includegraphics[width=\linewidth]{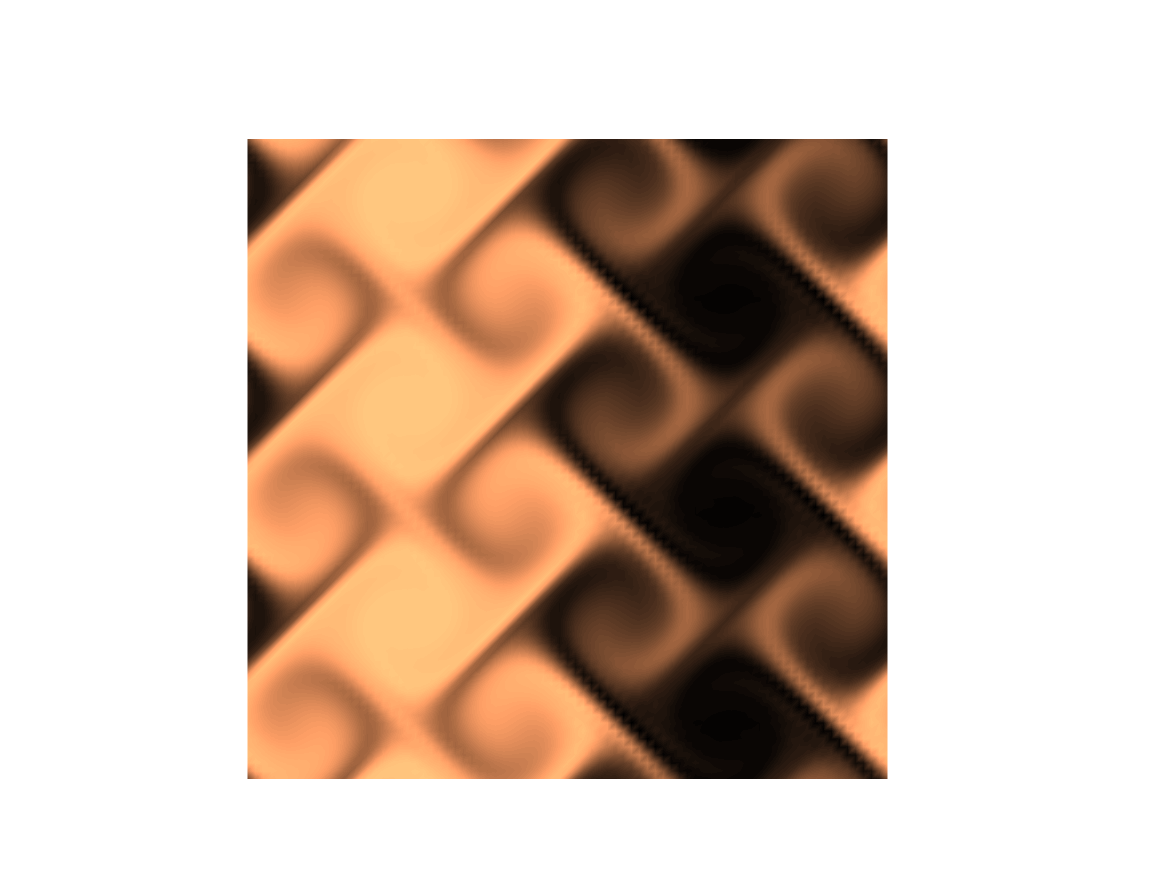}
  \end{subfigure}
    \begin{subfigure}[b]{0.24\linewidth}
    \includegraphics[width=\linewidth]{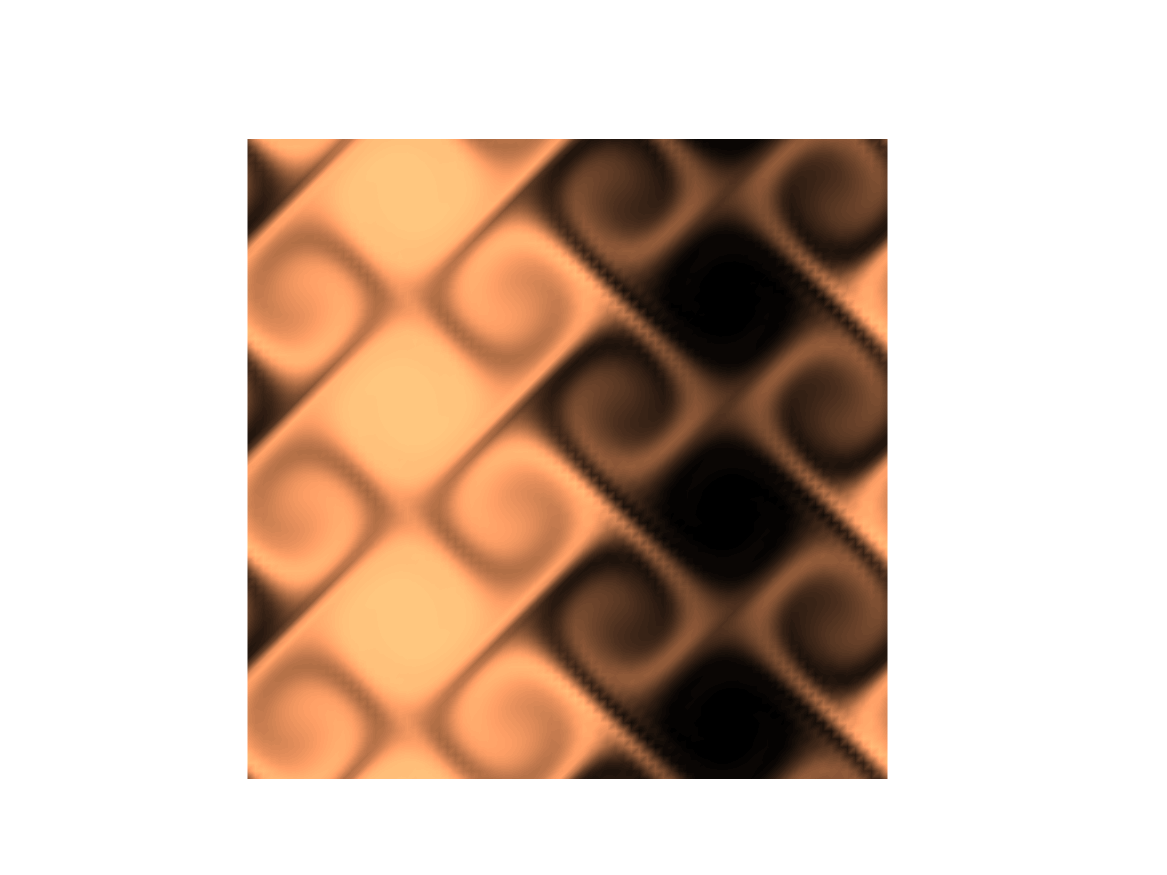}
  \end{subfigure}
  \begin{subfigure}[b]{0.24\linewidth}
    \includegraphics[width=\linewidth]{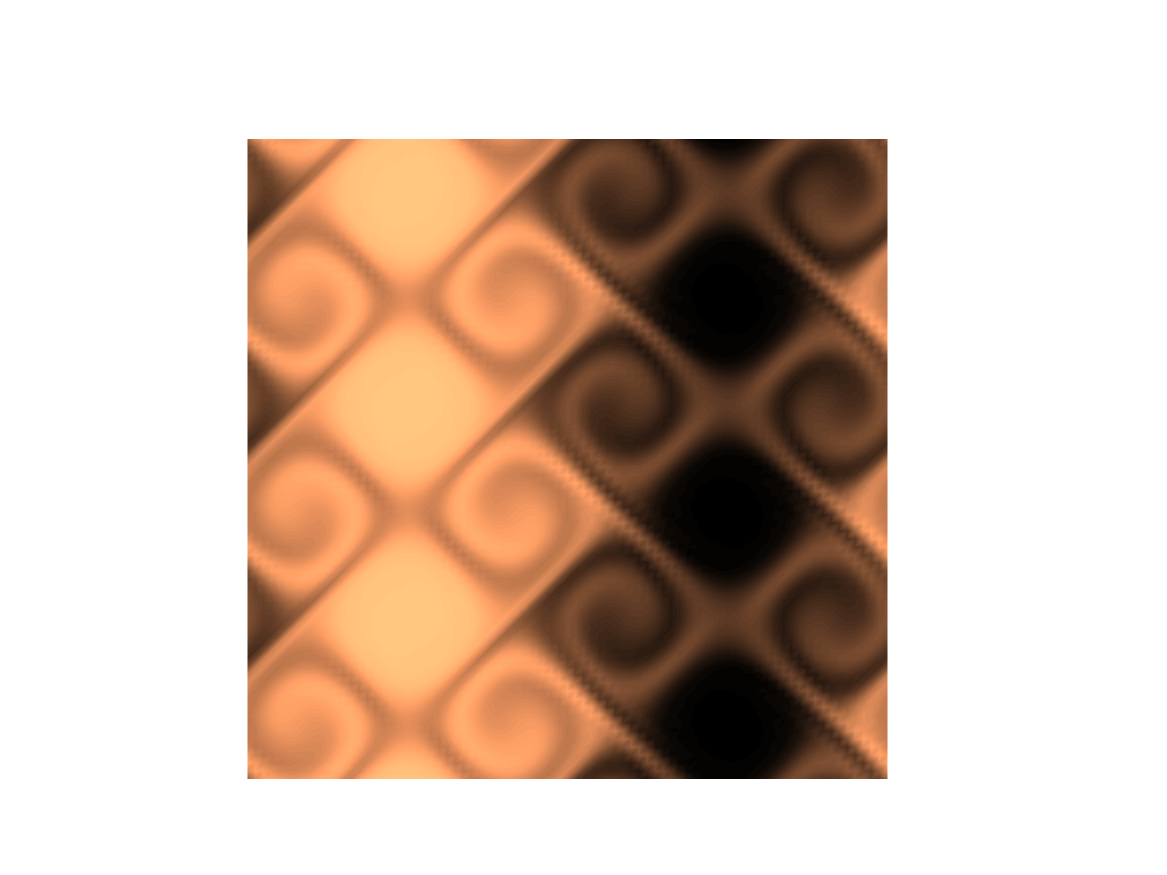}
  \end{subfigure}
  \begin{subfigure}[b]{0.24\linewidth}
    \includegraphics[width=\linewidth]{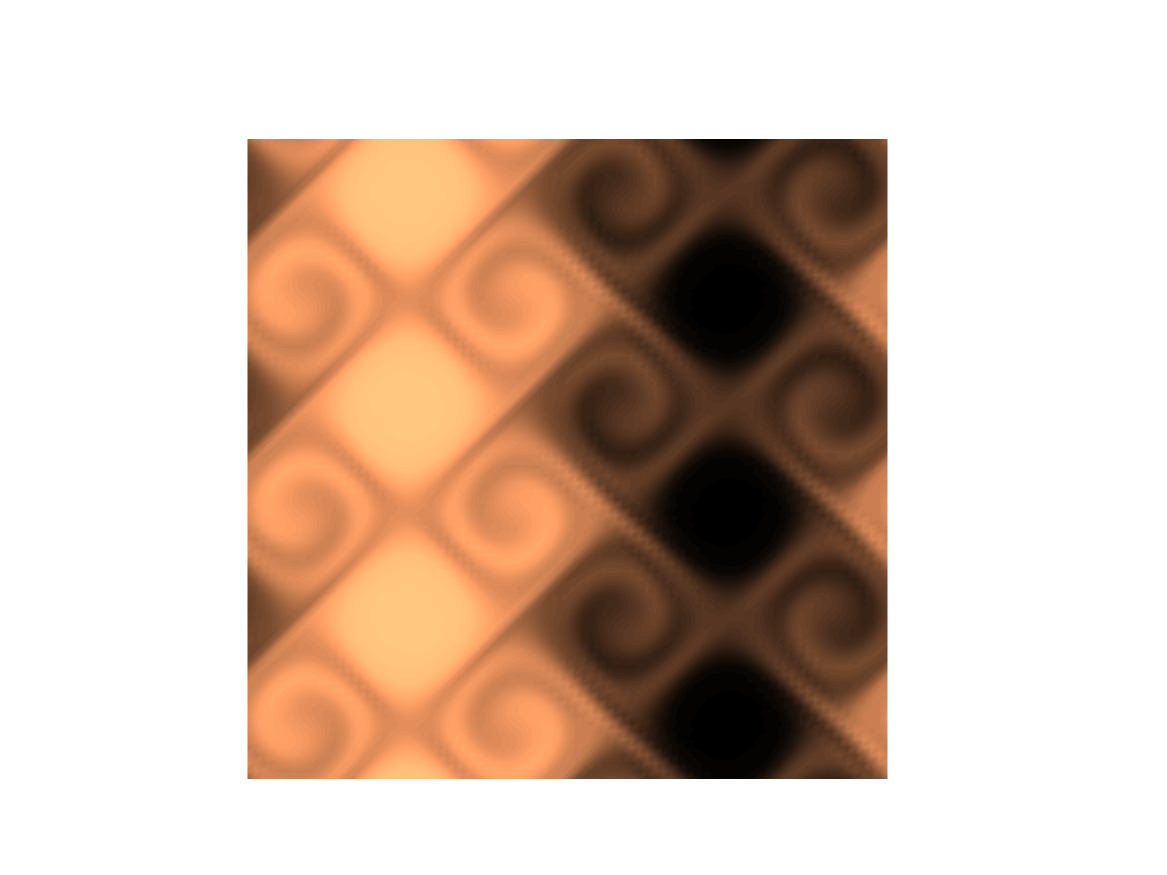}
  \end{subfigure}
  \begin{subfigure}[b]{0.24\linewidth}
    \includegraphics[width=\linewidth]{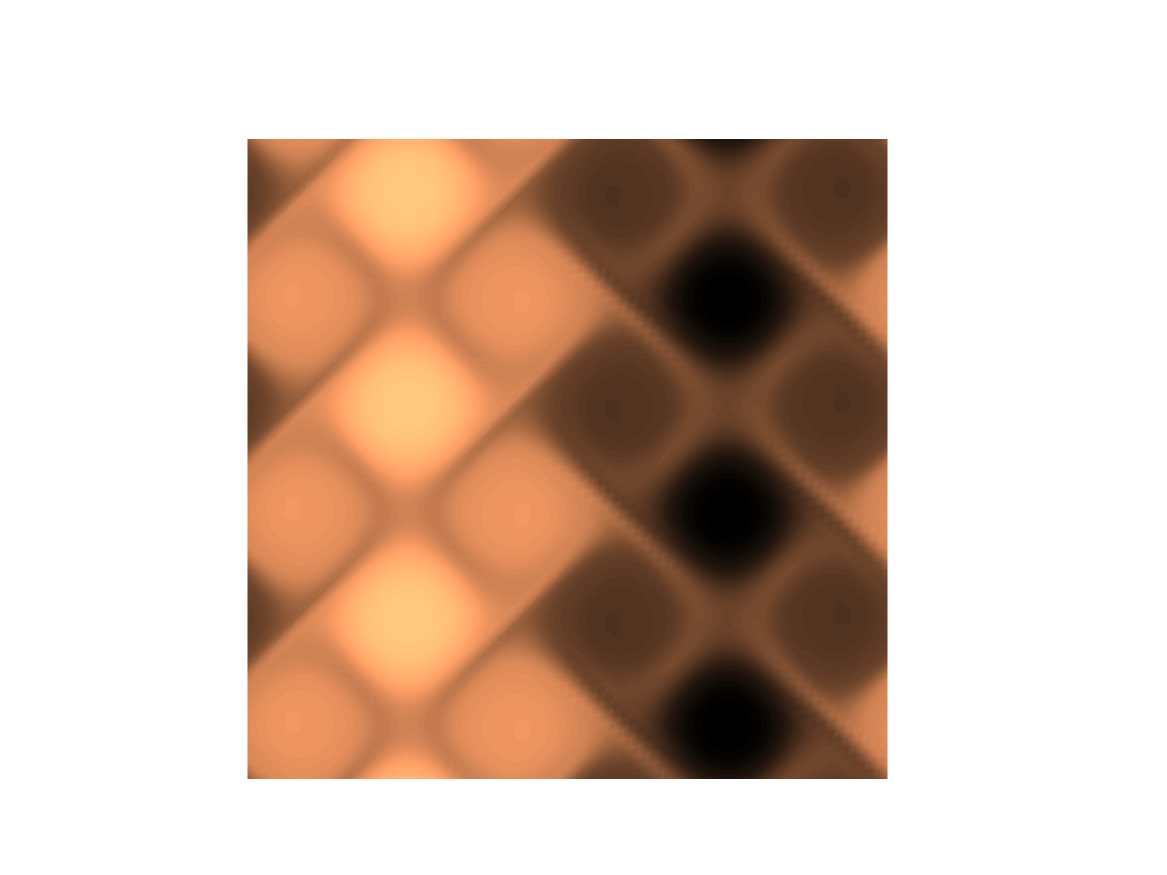}
  \end{subfigure}
  \caption{Snapshots of the evolutions of the solution $h$ to \eqref{FP:heps} when $\psi_\eps$ is given in \eqref{e:stream_function} with $\eps=1$.}
  \label{fig:cell1}
\end{figure}

\section{Conclusions}

The problem of the effect of compressible perturbations on the long time behaviour of solutions to the advection-diffusion equation was studied in this paper. In particular, for shear flows we characterize homogenization rates (see Theorem \ref{thm:main1}) and the longtime dynamics (see Theorem \ref{thm:main2}) of the solution very precisely in terms of the
dependence on the parameter (given by $1/\nu$) measuring the size of the incompressible perturbation. In the case of more complicated flows, such as
the Childress-Soward flow, we exhibit numerical evidence for what the expected behaviour should be.

There are several open questions that we plan to return to in future work. First obtaining similar sharp quantitative estimates for flows with closed streamlines is a challenging and interesting problem. This problem has already been studied for radial flows. Second, understanding in more depth the connection between the scaling of the effective diffusion coefficient with respect to the strength of the perturbation and the long time behaviour of solutions to the advection-diffusion equation. Third, studying the ``inverse'' problem, namely identifying the optimal nonreversible perturbation that maximizes the rate of convergence to equilibrium. Finally studying similar problems for random velocity fields. 

\medskip {\bf Acknowledgments}
M. Coti Zelati acknowledges funding from the Royal Society through a University Research Fellowship (URF\textbackslash R1\textbackslash 191492). The work of G.P. was partially funded by the EPSRC, grant number EP/P031587/1, and by J.P. Morgan Chase $\&$ Co. Any views or opinions expressed herein are solely those of the authors listed, and may differ from the views and opinions expressed by J.P. Morgan Chase $\&$ Co. or its affiliates. This material is not a product of the Research Department of J.P. Morgan Securities LLC. This material does not constitute a solicitation or offer in any jurisdiction.

\bibliographystyle{plain} 

\bibliography{biblio}

\begin{bibdiv}
\begin{biblist}

\bib{Bakry2014}{book}{
      author={Bakry, D.},
      author={Gentil, I.},
      author={Ledoux, M.},
       title={Analysis and geometry of {M}arkov diffusion operators},
      series={Grundlehren der Mathematischen Wissenschaften [Fundamental
  Principles of Mathematical Sciences]},
   publisher={Springer, Cham},
        date={2014},
      volume={348},
         url={http://dx.doi.org/10.1007/978-3-319-00227-9},
}

\bib{bhatta}{article}{
      author={Battacharya, R.},
       title={A central limit theorem for diffusions with periodic
  coefficients},
        date={1985},
     journal={The Annals of Probability},
      volume={13},
       pages={385\ndash 396},
}

\bib{BW13}{article}{
      author={Beck, Margaret},
      author={Wayne, C.~Eugene},
       title={Metastability and rapid convergence to quasi-stationary bar
  states for the two-dimensional {N}avier-{S}tokes equations},
        date={2013},
     journal={Proc. Roy. Soc. Edinburgh Sect. A},
      volume={143},
      number={5},
       pages={905\ndash 927},
         url={http://dx.doi.org/10.1017/S0308210511001478},
}

\bib{BCZ15}{article}{
      author={Bedrossian, Jacob},
      author={Coti~Zelati, Michele},
       title={Enhanced dissipation, hypoellipticity, and anomalous small noise
  inviscid limits in shear flows},
        date={2017},
     journal={Arch. Ration. Mech. Anal.},
      volume={224},
      number={3},
       pages={1161\ndash 1204},
}

\bib{lions}{book}{
      author={Bensoussan, A.},
      author={Lions, J.-L.},
      author={Papanicolaou, G.},
       title={Asymptotic analysis for periodic structures},
      series={Studies in Mathematics and its Applications},
   publisher={North-Holland Publishing Co.},
     address={Amsterdam},
        date={1978},
      volume={5},
}

\bib{CampPiatn2002}{incollection}{
      author={Campillo, F.},
      author={Piatnitski, A.},
       title={Effective diffusion in vanishing viscosity},
        date={2002},
   booktitle={Nonlinear partial differential equations and their applications.
  {C}oll\`ege de {F}rance {S}eminar, vol. xiv ({P}aris, 1997/1998)},
      series={Stud. Math. Appl.},
      volume={31},
   publisher={North-Holland},
     address={Amsterdam},
       pages={133\ndash 145},
}

\bib{CKRZ08}{article}{
      author={Constantin, P.},
      author={Kiselev, A.},
      author={Ryzhik, L.},
      author={Zlatos, A.},
       title={Diffusion and mixing in fluid flow},
        date={2008},
     journal={Ann. of Math. (2)},
      volume={168},
      number={2},
       pages={643\ndash 674},
         url={http://dx.doi.org/10.4007/annals.2008.168.643},
}

\bib{CZDE18}{article}{
      author={{Coti Zelati}, M.},
      author={{Delgadino}, M.G.},
      author={{Elgindi}, T.M.},
       title={On the relation between enhanced dissipation time-scales and
  mixing rates},
        date={2020},
     journal={Comm. Pure Appl. Math.},
      volume={73},
      number={6},
       pages={1205\ndash 1244},
}

\bib{CZ19}{article}{
      author={Coti~Zelati, Michele},
       title={Stable mixing estimates in the infinite {P}\'{e}clet number
  limit},
        date={2020},
     journal={J. Funct. Anal.},
      volume={279},
      number={4},
       pages={108562},
}

\bib{CZD19}{article}{
      author={{Coti Zelati}, Michele},
      author={{Dolce}, Michele},
       title={{Separation of time-scales in drift-diffusion equations on
  $\mathbb{R}^2$}},
        date={2019-07},
     journal={to appear in J. Math. Pures Appl., ArXiv e-prints},
      eprint={1907.04012},
}

\bib{CZDri19}{article}{
      author={{Coti Zelati}, Michele},
      author={{Drivas}, Theodore~D.},
       title={{A stochastic approach to enhanced diffusion}},
        date={2019-11},
     journal={to appear in Ann. Sc. Norm. Super. Pisa Cl. Sci., ArXiv
  e-prints},
      eprint={1911.09995},
}

\bib{DuncanLelievrePavliotis2016}{article}{
      author={Duncan, A.~B.},
      author={Leli{\`e}vre, T.},
      author={Pavliotis, G.~A.},
       title={Variance {R}eduction {U}sing {N}onreversible {L}angevin
  {S}amplers},
        date={2016},
     journal={J. Stat. Phys.},
      volume={163},
      number={3},
       pages={457\ndash 491},
         url={http://dx.doi.org/10.1007/s10955-016-1491-2},
}

\bib{Fannjiang02}{article}{
      author={Fannjiang, Albert},
       title={Time scales in homogenization of periodic flows with vanishing
  molecular diffusion},
        date={2002},
     journal={J. Differential Equations},
      volume={179},
      number={2},
       pages={433\ndash 455},
         url={https://doi.org/10.1006/jdeq.2001.4039},
}

\bib{FI19}{article}{
      author={Feng, Yuanyuan},
      author={Iyer, Gautam},
       title={Dissipation enhancement by mixing},
        date={2019},
     journal={Nonlinearity},
      volume={32},
      number={5},
       pages={1810\ndash 1851},
}

\bib{HP08}{article}{
      author={Hairer, M.},
      author={Pavliotis, G.},
       title={From ballistic to diffusive behavior in periodic potentials},
        date={2008},
     journal={J. Stat. Phys.},
      volume={131},
       pages={175\ndash 202},
}

\bib{HP04}{article}{
      author={Hairer, M.},
      author={Pavliotis, G.~A.},
       title={Periodic homogenization for hypoelliptic diffusions},
        date={2004},
     journal={J. Statist. Phys.},
      volume={117},
      number={1-2},
       pages={261\ndash 279},
}

\bib{Hwang_al1993}{article}{
      author={Hwang, C.-R.},
      author={Hwang-Ma, S.-Y.},
      author={Sheu, S.-J.},
       title={Accelerating {G}aussian diffusions},
        date={1993},
     journal={Ann. Appl. Probab.},
      volume={3},
      number={3},
       pages={897\ndash 913},
}

\bib{Hwang_al2005}{article}{
      author={Hwang, C.-R.},
      author={Hwang-Ma, S.-Y.},
      author={Sheu, S.-J.},
       title={Accelerating diffusions},
        date={2005},
     journal={Ann. Appl. Probab.},
      volume={15},
      number={2},
       pages={1433\ndash 1444},
}

\bib{IXZ19}{article}{
      author={{Iyer}, Gautam},
      author={{Xu}, Xiaoqian},
      author={{Zlato{\v{s}}}, Andrej},
       title={{Convection-Induced Singularity Suppression in the Keller-Segel
  and Other Non-linear PDEs}},
        date={2019-08},
     journal={arXiv e-prints},
      eprint={1908.01941},
}

\bib{KSh91}{book}{
      author={Karatzas, I.},
      author={Shreve, S.E.},
       title={Brownian {M}otion and {S}tochastic {C}alculus},
     edition={Second},
      series={Graduate Texts in Mathematics},
   publisher={Springer-Verlag},
     address={New York},
        date={1991},
      volume={113},
}

\bib{LelievreNierPavliotis2013}{article}{
      author={Lelievre, T.},
      author={Nier, F.},
      author={Pavliotis, G.~A.},
       title={Optimal non-reversible linear drift for the convergence to
  equilibrium of a diffusion},
        date={{2013}},
     journal={J. Stat. Phys.},
      volume={{152}},
      number={2},
       pages={{ 237\ndash 274 }},
}

\bib{MajMcL93}{article}{
      author={Majda, A.J.},
      author={McLaughlin, R.M.},
       title={The effect of mean flows on enhanced diffusivity in transport by
  incompressible periodic velocity fields},
        date={1993},
     journal={Stud. Appl. Math.},
      volume={89},
      number={3},
       pages={245\ndash 279},
}

\bib{kramer}{article}{
      author={Majda, Andrew~J.},
      author={Kramer, Peter~R.},
       title={Simplified models for turbulent diffusion: theory, numerical
  modelling, and physical phenomena},
        date={1999},
     journal={Phys. Rep.},
      volume={314},
      number={4-5},
       pages={237\ndash 574},
}

\bib{McLaughlinForest99}{article}{
      author={McLaughlin, R.~M.},
      author={Forest, M.~G.},
       title={An anelastic, scale-separated model for mixing, with application
  to atmospheric transport phenomena},
        date={1999},
     journal={Phys. Fluids},
      volume={11},
      number={4},
       pages={880\ndash 892},
         url={http://dx.doi.org/10.1063/1.869967},
}

\bib{McL98}{article}{
      author={McLaughlin, R.M.},
       title={Numerical averaging and fast homogenization},
        date={1998},
     journal={J. Statist. Phys.},
      volume={90},
      number={3-4},
       pages={597\ndash 626},
}

\bib{Golden_al_2020}{article}{
      author={Murphy, N.~B.},
      author={Cherkaev, E.},
      author={Zhu, J.},
      author={Xin, J.},
      author={Golden, K.~M.},
       title={Spectral analysis and computation for homogenization of advection
  diffusion processes in steady flows},
        date={2020},
     journal={J. Math. Phys.},
      volume={61},
      number={1},
       pages={013102, 34},
         url={https://doi.org/10.1063/1.5127457},
}

\bib{Pap95}{incollection}{
      author={Papanicolaou, G.~C.},
       title={Diffusion in random media},
        date={1995},
   booktitle={Surveys in applied mathematics, vol.\ 1},
      series={Surveys Appl. Math.},
      volume={1},
   publisher={Plenum},
     address={New York},
       pages={205\ndash 253},
}

\bib{pardoux}{article}{
      author={Pardoux, E.},
       title={Homogenization of linear and semilinear second order parabolic
  pdes with periodic coefficients: A probabilistic approach},
        date={1999},
     journal={Journal of Functional Analysis},
      volume={167},
       pages={498\ndash 520},
}

\bib{thesis}{book}{
      author={Pavliotis, G.~A.},
       title={Homogenization theory for advection-diffusion equations with mean
  flow},
   publisher={ProQuest LLC, Ann Arbor, MI},
        date={2002},
  url={http://gateway.proquest.com/openurl?url_ver=Z39.88-2004&rft_val_fmt=info:ofi/fmt:kev:mtx:dissertation&res_dat=xri:pqdiss&rft_dat=xri:pqdiss:3057679},
        note={Thesis (Ph.D.)--Rensselaer Polytechnic Institute},
}

\bib{Pavliotis2010}{article}{
      author={Pavliotis, G.~A.},
       title={Asymptotic analysis of the {G}reen-{K}ubo formula},
        date={2010},
     journal={IMA J. Appl. Math.},
      volume={75},
      number={6},
       pages={951\ndash 967},
         url={http://dx.doi.org/10.1093/imamat/hxq039},
}

\bib{PavSt05b}{article}{
      author={Pavliotis, G.~A.},
      author={Stuart, A.~M.},
       title={Periodic homogenization for inertial particles},
        date={2005},
     journal={Phys. D},
      volume={204},
      number={3-4},
       pages={161\ndash 187},
}

\bib{PavlStuZyg07}{article}{
      author={Pavliotis, G.~A.},
      author={Stuart, A.~M.},
      author={Zygalakis, K.~C.},
       title={Homogenization for inertial particles in a random flow},
        date={2007},
     journal={Commun. Math. Sci.},
      volume={5},
      number={3},
       pages={507\ndash 531},
}

\bib{PavlStZyg09}{article}{
      author={Pavliotis, G.~A.},
      author={Stuart, A.~M.},
      author={Zygalakis, K.~C.},
       title={Calculating effective diffusiveness in the limit of vanishing
  molecular diffusion},
        date={2009},
     journal={J. Comput. Phys.},
      volume={228},
      number={4},
       pages={1030\ndash 1055},
         url={https://doi.org/10.1016/j.jcp.2008.10.014},
}

\bib{PavlStBan06}{incollection}{
      author={Pavliotis, G.A.},
      author={A.M.~Stuart, A.~M.},
      author={Band, L.},
       title={Monte {C}arlo studies of effective diffusivities for inertial
  particles},
        date={2006},
   booktitle={Monte carlo and quasi-monte carlo methods 2004},
   publisher={Springer},
     address={Berlin},
       pages={431\ndash 441},
}

\bib{pavliotis2014book}{book}{
      author={Pavliotis, Grigorios},
       title={Stochastic processes and applications: Diffusion processes, the
  fokker-planck and langevin equations},
   publisher={Springer New York},
        date={2014},
}

\bib{PavlSt08}{book}{
      author={Pavliotis, Grigorios~A.},
      author={Stuart, Andrew~M.},
       title={Multiscale methods},
      series={Texts in Applied Mathematics},
   publisher={Springer, New York},
        date={2008},
      volume={53},
}

\bib{RV20}{article}{
      author={{Renaud}, Antoine},
      author={{Vanneste}, Jacques},
       title={{Dispersion of inertial particles in cellular flows in the
  small-Stokes, large-P{\'e}clet regime}},
        date={2020-01},
     journal={arXiv e-prints},
      eprint={2001.10618},
}

\bib{vergassola}{article}{
      author={Vergassola, M.},
      author={Avellaneda, M.},
       title={Scalar transport in compressible flow},
        date={1997},
        ISSN={0167-2789},
     journal={Phys. D},
      volume={106},
      number={1-2},
       pages={148\ndash 166},
}

\bib{Villani09}{article}{
      author={Villani, C{\'e}dric},
       title={Hypocoercivity},
        date={2009},
     journal={Mem. Amer. Math. Soc.},
      volume={202},
      number={950},
       pages={iv+141},
}

\bib{WEI18}{article}{
      author={{Wei}, Dongyi},
       title={{Diffusion and mixing in fluid flow via the resolvent estimate}},
        date={2018-11},
     journal={arXiv e-prints},
      eprint={1811.11904},
}

\bib{Zlatos2010}{article}{
      author={Zlato\v{s}, Andrej},
       title={Diffusion in fluid flow: dissipation enhancement by flows in
  2{D}},
        date={2010},
     journal={Comm. Partial Differential Equations},
      volume={35},
      number={3},
       pages={496\ndash 534},
}

\end{biblist}
\end{bibdiv}

\end{document}